\documentclass[a4paper,reqno]{amsart}

\usepackage{graphicx}
\usepackage{amsmath}
\usepackage{amssymb}
\input xy
\xyoption{all}

\usepackage{color}

\numberwithin{equation}{section}

\theoremstyle{plain}

\newtheorem{thm}{Theorem}[section]
\newtheorem{cor}[thm]{Corollary}
\newtheorem{lem}[thm]{Lemma}
\newtheorem{prop}[thm]{Proposition}

\newtheorem{defn}[thm]{Definition}
\newtheorem{exm}[thm]{Example}

\theoremstyle{remark}
\newtheorem{rem}[thm]{Remark}

\newdir{ >}{{}*!/-5pt/\dir{>}}

\renewcommand{\mod}{\operatorname{mod}\nolimits}

\newcommand{\add}{\operatorname{add}\nolimits}

\newcommand{\Hom}{\operatorname{Hom}\nolimits}

\newcommand{\End}{\operatorname{End}\nolimits}

\newcommand{\Ker}{\operatorname{Ker}\nolimits}

\newcommand{\op}{\operatorname{op}\nolimits}

\newcommand{\del}{\delta}

\newcommand{\Cone}{\operatorname{Cone}\nolimits}
\newcommand{\CoCone}{\operatorname{CoCone}\nolimits}

\newcommand{\I}{\mathcal I}

\newcommand{\B}{\mathcal B}

\newcommand{\U}{\mathcal U}
\newcommand{\V}{\mathcal V}

\newcommand{\h}{\mathcal H}

\newcommand{\T}{\mathcal T}
\newcommand{\D}{\mathcal D}
\newcommand{\J}{\mathcal J}
\newcommand{\K}{\mathcal K}

\newcommand{\C}{\mathcal C}

\newcommand{\EE}{\mathbb E}

\newcommand{\svecv}[2]{\left(\begin{smallmatrix}
      #1 \\
      #2
    \end{smallmatrix}\right)}

\newcommand{\svech}[2]{\left(\begin{smallmatrix}
      #1 & #2
\end{smallmatrix}\right)}

\def\Ab{\mathsf{Ab}}
\renewcommand{\emph}{\textit}
\renewcommand{\phi}{\varphi}

\begin{document}

\title{Abelian categories arising from cluster-tilting subcategories II: quotient functors}\footnote{The first author is supported by the Fundamental Research Funds for the Central Universities (Grants No.2682018ZT25). The second author is supported by the Hunan Provincial Natural Science Foundation of China (Grants No.2018JJ3205) and the NSF of China (Grants No.11671221)}
\author{Yu Liu and Panyue Zhou}
\address{Department of Mathematics, Southwest Jiaotong University, 610031, Chengdu, Sichuan, People's Republic of China}
\email{liuyu86@swjtu.edu.cn}
\address{College of Mathematics, Hunan Institute of Science and Technology, 414006, Yueyang, Hunan, People's Republic of China}
\email{panyuezhou@163.com}
\thanks{The authors wish to thank Professor Bin Zhu for their helpful advices.}
\begin{abstract}
In this paper, we consider a kind of ideal quotient of an extriangulated category such that the ideal is the kernel of a functor from this extriangulated category to an abelian category. We study a condition when the functor is dense and full, in another word, the ideal quotient becomes abelian. Moreover, a new equivalent characterization of cluster-tilting subcategories is given by applying homological methods according to this functor. As an application, we show that in a connected $2$-Calabi-Yau triangulated category $\B$, a functorially finite, extension closed subcategory $\T$ of $\B$ is cluster tilting if and only if $\B/\T$ is an abelian category.
\end{abstract}

\maketitle

\section{Introduction}
Cluster tilting theory sprouted from the categorification of Fomin-Zelevinsky's cluster algebras.
It is used to construct abelian categories from some triangulated categories and exact categories. Koenig and Zhu \cite[Theorem 3.3]{KZ} showed that the quotient of any triangulated category modulo a
cluster-tilting subcategory is abelian. A similar result on exact category was shown by Demonet and Liu in \cite{DL}.

However not all triangulated categories or exact categories contain a cluster-tilting subcategory, but they may still admit abelian quotients (see Section 5 in \cite{KZ}, Example 18 in \cite{GJ} and Example \ref{ex2}). Moreover, not all abelian quotients are arising by moduloing subcategories (see Exmaple 19 in \cite{GJ} and Example \ref{ex1}). In fact, for example (see \cite{GJ}), in a Hom-finite, Krull-Schmidt triangulated $k$-category $\T$, where $k$ is a field, let $\mathcal J$ be an ideal such that $\mathcal J (A,B)=\{f\in \Hom_{\T}(A,B) \text{ }|\text{ } \Hom_{\T}(T,f) =0\}$. Then if $\Hom_{\T}(T,-)$ is a quotient functor (i.e. dense and full), we have the following commutative diagram
$$\xymatrix@C=0.5cm@R0.5cm{
\T \ar[rr]^-{\Hom_{\T}(T,-)} \ar[dr]_{\pi} && \mod \End_{\T}(T)^{\op}\\
&\T/\mathcal J \ar[ur]_-{\simeq}
}
$$
where $\pi:\T\to \T/\mathcal J$ is the natural quotient functor. Grimeland and Jacobsen showed the following result, which contains all the cases above on triangulated categories.

\begin{thm}\cite[Theorem 17]{GJ} and \cite[Theorem 25]{GJ}
Let $\T$ be a Hom-finite Krull-Schmidt triangulated $k$-category, where $k$ is a field. Then
$\Hom_{\T}(T,-)$ is a quotient functor, i.e. full and dense, if and only if the following two conditions are satisfied
\begin{itemize}
\item[(a)] For all right minimal morphisms $T_1\to T_0$, where $T_0,T_1\in\add T$, all triangles
$$T_1\rightarrow T_0 \rightarrow X\xrightarrow{~h~}\Sigma T_1$$
satisfy $\Hom_{\T}(T,h)=0$.

\item[(b)] For all indecomposable $T$-support objects $X$, i.e. $\Hom_{\T}(T,X)\neq 0$, there exist triangles
$$T_1\rightarrow T_0 \rightarrow X\xrightarrow{~h~}\Sigma T_1$$
with $T_1,T_0\in\add T$ and $\Hom_{\T}(T,h)= 0$.
\end{itemize}
\medskip
An object $T\in\T$ is a cluster-tilting object if and only if $(a)$ is satisfied
and furthermore:
\begin{itemize}
\item[(c)] For all indecomposable  objects $X$, there exists a triangle
$$T_1\rightarrow T_0 \rightarrow X\xrightarrow{~h~}\Sigma T_1$$
with $T_1,T_0\in\add T$ and $\Hom_{\T}(T,h)\neq 0$.

\item[(d)] if $T'$ is an indecomposable summand of $T$, then $\Sigma T'\notin\add T$.
\end{itemize}
\end{thm}

It is very natural to consider if the similar result holds on exact categories. Also we are considering if we can drop some assumptions for the triangulated case \cite{GJ}. Hence in this article we will still work on an extriangulated category as in \cite{LZ} and consider the similar question on it. The notion of an extriangulated category was introduced in \cite{NP} (please see
section 2 for detailed definition of extriangulated category), which is a simultaneous generalization of
exact category and triangulated category.

In this article, let $(\B,\EE,\mathfrak{s})$ be a Hom-finite Krull-Schmidt extriangulated category over a field $k$. Any subcategory we discuss in this article will be full and closed under isomorphisms. We also assume that $\B$ has enough projectives $\mathcal P$ and enough injectives $\mathcal I$.

We first introduce some notions.

We denote by $\B/\D$ the category whose objects are objects of $\B$ and whose morphisms are elements of
$\Hom_{\B}(A,B)/\D(A,B)$ for $A,B\in\B$, where $\D(A,B)$ the subgroup of $\Hom_{\B}(A,B)$ consisting of morphisms
which factor through an object in $\D$. Such category is called the quotient category
of $\B$ by $\D$.

When $\D=\mathcal P$, for any morphism $f\colon A\to B$ in $\B$, we denote by $\underline{f}$ the image of $f$ under
the natural quotient functor $\B\to \B/\mathcal P$. We also denote $\B'/\mathcal P$ by $\underline \B'$ if $\mathcal P\subseteq \B'\subseteq \B$.

When $\D=\mathcal I$, for any morphism $f\colon A\to B$ in $\B$, we denote by $\overline{f}$ the image of $f$ under
the natural quotient functor $\B\to \B/\mathcal I$. We also denote $\B'/\mathcal I$ by $\overline {\B'}$ if $\mathcal I\subseteq \B'\subseteq \B$.

Let $\B',\B''$ be two subcategories of $\B$, denote by $\CoCone(\B',\B'')$ the subcategory of objects $X$ admitting an $\EE$-triangle $\xymatrix{ X\ar[r] &B'\ar[r]  &B'' \ar@{-->}[r] &}$ where $B'\in \B'$ and $B''\in \B''$. We denote by $\Cone(\B',\B'')$ the subcategory of objects $Y$ admitting an $\EE$-triangle $\xymatrix{ B'\ar[r]  &B'' \ar[r] &Y \ar@{-->}[r] &}$ where $B'\in \B'$ and $B''\in \B''$.

Let $\Omega \B'=\CoCone(\mathcal P,\B')$ and $\Sigma \B'=\Cone(\B',\mathcal I)$. If there is no confusion, we can write an object $D$ in the form $\Omega B$ if it admits an $\EE$-triangle $\xymatrix{D \ar[r] &P \ar[r] &B \ar@{-->}[r] &}$ where $P\in \mathcal P$, and write an object $D'$ in the form $\Sigma B'$ if it admits an $\EE$-triangle $\xymatrix{B' \ar[r] &I \ar[r] &D' \ar@{-->}[r] &}$ where $I\in \mathcal I$.

From now on, let $\T$ be a subcategory of $\B$ which is closed under direct sums and summands. We give the first main result of this article, which is an extriangulated category version of \cite[Theorem 17]{GJ}.

\begin{thm}\label{main1}
Let $\T$ be a contravariantly finite subcategory of $\B$ such that $\mathcal P\subset \T$ and $\EE(\T,\mathcal P)=0$. Then the functor $\Hom_{\underline \B}(\Omega \T,-)$ is a quotient functor, \emph{i.e.} full and dense, if and only if the following conditions are satisfied:
\begin{itemize}
\item[(a)] In the following commutative diagram
$$\xymatrix{
R_1 \ar@{=}[r] \ar[d]_f &R_1 \ar[d]\\
R_2 \ar[d]_g \ar[r] &P \ar[r] \ar[d] &T_2 \ar@{=}[d] \ar@{-->}[r]&\\
A \ar[r]_h \ar@{-->}[d] &T_1 \ar[r] \ar@{-->}[d] &T_2 \ar@{-->}[r] &\\
&& &&
}
$$
where $T_1,T_2\in \T$ and $P\in \mathcal P$, if $0\neq \underline f$ is right minimal, then $\Hom_{\underline \B}(\Omega \T,\underline h)=0$.
\item[(b)] For any indecomposable object $A$ such that $\Hom_{\underline \B}(\Omega \T,A)\neq 0$, $A$ admits a commutative diagram
$$\xymatrix{
R_1 \ar@{=}[r] \ar[d]_f &R_1 \ar[d]\\
R_2 \ar[d]_g \ar[r] &P \ar[r] \ar[d] &T_2 \ar@{=}[d] \ar@{-->}[r]&\\
A \ar[r]_h \ar@{-->}[d] &T_1 \ar[r] \ar@{-->}[d] &T_2 \ar@{-->}[r] &\\
&& &&
}
$$
where $T_1,T_2\in \T$, $P\in \mathcal P$ and $\Hom_{\underline \B}(\Omega \T,\underline h)=0$.
\end{itemize}

\end{thm}

By \cite[Proposition 3.30]{NP}, $\B/(\mathcal P\cap \mathcal I)$  is still an extriangulated category. We recall that an extriangulated category is called connected if it can not be written as a direct sum of two extriangulated subcategories (for details, please see \cite[Definition 1.9]{LZ}). We apply Theorem \ref{main1} and prove the second main result of this article.

\begin{thm}\label{main}
Let $\T$ be a non-zero functorially finite subcategory of $\B$ such that
\begin{itemize}
\item[(1)] $\mathcal P\subset \T$ and $\mathcal I\subset \T$;
\item[(2)] $\EE(\T,\mathcal P)=0=\EE(\mathcal I,\T)$;
\item[(3)] For object $B$, $\Hom_{\underline \B}(\Omega \T,B)=0$ if and only if $\Hom_{\overline \B}(B,\Sigma \T)=0$.
\end{itemize}
If $\B/(\mathcal P\cap \mathcal I)$ is connected, $\Hom_{\underline{\B}}(\Omega \T,-)$ and $\Hom_{\overline{\B}}(-,\Sigma \T)$ are quotient functors, then $\T$ is a cluster-tilting subcategory if and only if one of the following conditions is satisfied
\begin{itemize}
\item[(a)] $\T\cap \Omega \T=\mathcal P$.
\item[(b)] $\T$ is extension closed.
\end{itemize}
\end{thm}

By using this theorem, we get the following interesting corollaries, the second one generalizes \cite[Theorem 27]{GJ}.

\begin{cor}
Let $\B$ be a connected, Krull-Schmidt, $2$-Calabi-Yau triangulated category over a field $k$, $\T$ be a functorially finite, extension closed subcategory. Then $\B/\T$ is abelian if and only if $\T$ is cluster-tilting.
\end{cor}

\begin{cor}
Let $\B$ be a connected, Krull-Schmidt, Hom-finte, $k$-linear, 2-Calabi-Yau triangulated category with suspension functor $\Sigma$, $\T=\add T$ be a functorially finite subcategory. If $\Hom_{\B}(T,-)$ is full and dense, then $T$ is a cluster tilting object or $\End_{\B}(T)^{\op}\simeq k$.
\end{cor}

This article is organized as follows. In Section 2, we review some elementary definitions and facts of
extriangulated category that we need. In Section 3, we prove our first main result and give some examples. In Section 4, we prove our second main result.

\section{Preliminaries}
Let us briefly recall the definition and basic properties of extriangulated categories from \cite{NP}. Throughout this paper, we assume that $\B$ is an additive category.

\begin{defn}
Suppose that $\B$ is equipped with an additive bifunctor $\mathbb{E}\colon\B^\mathrm{op}\times\B\to\Ab$, where $\Ab$ is the category of abelian groups. For any pair of objects $A,C\in\B$, an element $\delta\in\mathbb{E}(C,A)$ is called an {\it $\mathbb{E}$-extension}. Thus formally, an $\EE$-extension is a triplet $(A,\delta,C)$.
For any $A,C\in\C$, the zero element $0\in\EE(C,A)$ is called the \emph{spilt $\EE$-extension}.

Let $\delta\in\mathbb{E}(C,A)$ be any $\mathbb{E}$-extension. By the functoriality, for any $a\in\B(A,A^{\prime})$ and $c\in\B(C^{\prime},C)$, we have $\mathbb{E}$-extensions
\[ \mathbb{E}(C,a)(\delta)\in\mathbb{E}(C,A^{\prime})\ \ \text{and}\ \ \mathbb{E}(c,A)(\delta)\in\mathbb{E}(C^{\prime},A). \]
We abbreviately denote them by $a_{\ast}\delta$ and $c^{\ast}\delta$.
In this terminology, we have
\[ \mathbb{E}(c,a)(\delta)=c^{\ast} a_{\ast}\delta=a_{\ast} c^{\ast}\delta \]
in $\mathbb{E}(C^{\prime},A^{\prime})$.

\end{defn}

\begin{defn}
Let $\delta\in\mathbb{E}(C,A)$ and $\delta^{\prime}\in\mathbb{E}(C^{\prime},A^{\prime})$ be two pair of $\mathbb{E}$-extensions. A {\it morphism} $(a,c)\colon\delta\to\delta^{\prime}$ of $\mathbb{E}$-extensions is a pair of morphisms $a\in\B(A,A^{\prime})$ and $c\in\B(C,C^{\prime})$ in $\B$, satisfying the equality
\[ a_{\ast}\delta=c^{\ast}\delta^{\prime}. \]
We simply denote it as $(a,c)\colon\delta\to\delta^{\prime}$.
\end{defn}

\begin{defn}
Let $\delta=(A,\delta,C)$ and $\delta^{\prime}=(A^{\prime},\delta^{\prime},C^{\prime})$ be any pair of $\mathbb{E}$-extensions. Let
\[ C\xrightarrow{~\iota_C~}C\oplus C^{\prime}\xleftarrow{~\iota_{C^{\prime}}~}C^{\prime} \]
and
\[ A\xrightarrow{~p_A~}A\oplus A^{\prime}\xleftarrow{~p_{A^{\prime}}~}A^{\prime} \]
be coproduct and product in $\B$, respectively. Remark that, by the additivity of $\mathbb{E}$, we have a natural isomorphism
\[ \mathbb{E}(C\oplus C^{\prime},A\oplus A^{\prime})\simeq \mathbb{E}(C,A)\oplus\mathbb{E}(C,A^{\prime})\oplus\mathbb{E}(C^{\prime},A)\oplus\mathbb{E}(C^{\prime},A^{\prime}). \]

Let $\delta\oplus\delta^{\prime}\in\mathbb{E}(C\oplus C^{\prime},A\oplus A^{\prime})$ be the element corresponding to $(\delta,0,0,\delta^{\prime})$ through this isomorphism. This is the unique element which satisfies
$$
\mathbb{E}(\iota_C,p_A)(\delta\oplus\delta^{\prime})=\delta,\ \mathbb{E}(\iota_C,p_{A^{\prime}})(\delta\oplus\delta^{\prime})=0,\
\mathbb{E}(\iota_{C^{\prime}},p_A)(\delta\oplus\delta^{\prime})=0,\ \mathbb{E}(\iota_{C^{\prime}},p_{A^{\prime}})(\delta\oplus\delta^{\prime})=\delta^{\prime}.
$$
\end{defn}

\begin{defn}
Let $A,C\in\B$ be any pair of objects. Two sequences of morphisms in $\B$
\[ A\overset{x}{\longrightarrow}B\overset{y}{\longrightarrow}C\ \ \text{and}\ \ A\overset{x^{\prime}}{\longrightarrow}B^{\prime}\overset{y^{\prime}}{\longrightarrow}C \]
are said to be {\it equivalent} if there exists an isomorphism $b\in\B(B,B^{\prime})$ which makes the following diagram commutative.
\[
\xy
(-16,0)*+{A}="0";
(3,0)*+{}="1";
(0,8)*+{B}="2";
(0,-8)*+{B^{\prime}}="4";
(-3,0)*+{}="5";
(16,0)*+{C}="6";
{\ar^{x} "0";"2"};
{\ar^{y} "2";"6"};
{\ar_{x^{\prime}} "0";"4"};
{\ar_{y^{\prime}} "4";"6"};
{\ar^{b}_{\simeq} "2";"4"};
{\ar@{}|{} "0";"1"};
{\ar@{}|{} "5";"6"};
\endxy
\]

We denote the equivalence class of $A\overset{x}{\longrightarrow}B\overset{y}{\longrightarrow}C$ by $[A\overset{x}{\longrightarrow}B\overset{y}{\longrightarrow}C]$.
\end{defn}

\begin{defn}
$\ \ $
\begin{enumerate}
\item[(1)] For any $A,C\in\B$, we denote as
\[ 0=[A\overset{\Big[\raise1ex\hbox{\leavevmode\vtop{\baselineskip-8ex \lineskip1ex \ialign{#\crcr{$\scriptstyle{1}$}\crcr{$\scriptstyle{0}$}\crcr}}}\Big]}{\longrightarrow}A\oplus C\overset{[0\ 1]}{\longrightarrow}C]. \]

\item[(2)] For any $[A\overset{x}{\longrightarrow}B\overset{y}{\longrightarrow}C]$ and $[A^{\prime}\overset{x^{\prime}}{\longrightarrow}B^{\prime}\overset{y^{\prime}}{\longrightarrow}C^{\prime}]$, we denote as
\[ [A\overset{x}{\longrightarrow}B\overset{y}{\longrightarrow}C]\oplus [A^{\prime}\overset{x^{\prime}}{\longrightarrow}B^{\prime}\overset{y^{\prime}}{\longrightarrow}C^{\prime}]=[A\oplus A^{\prime}\overset{x\oplus x^{\prime}}{\longrightarrow}B\oplus B^{\prime}\overset{y\oplus y^{\prime}}{\longrightarrow}C\oplus C^{\prime}]. \]
\end{enumerate}
\end{defn}

\begin{defn}
Let $\mathfrak{s}$ be a correspondence which associates an equivalence class $\mathfrak{s}(\delta)=[A\overset{x}{\longrightarrow}B\overset{y}{\longrightarrow}C]$ to any $\mathbb{E}$-extension $\delta\in\mathbb{E}(C,A)$. This $\mathfrak{s}$ is called a {\it realization} of $\mathbb{E}$, if it satisfies the following condition $(\star)$. In this case, we say that the sequence $A\overset{x}{\longrightarrow}B\overset{y}{\longrightarrow}C$ {\it realizes} $\delta$, whenever it satisfies $\mathfrak{s}(\delta)=[A\overset{x}{\longrightarrow}B\overset{y}{\longrightarrow}C]$.
\begin{itemize}
\item[$(\star)$] Let $\delta\in\mathbb{E}(C,A)$ and $\delta^{\prime}\in\mathbb{E}(C^{\prime},A^{\prime})$ be any pair of $\mathbb{E}$-extensions, with
\[\mathfrak{s}(\delta)=[A\overset{x}{\longrightarrow}B\overset{y}{\longrightarrow}C]\text{ and } \mathfrak{s}(\delta^{\prime})=[A^{\prime}\overset{x^{\prime}}{\longrightarrow}B^{\prime}\overset{y^{\prime}}{\longrightarrow}C^{\prime}].\]
Then, for any morphism $(a,c)\colon\delta\to\delta^{\prime}$, there exists $b\in\B(B,B^{\prime})$ which makes the following diagram commutative.
$$
\xy
(-12,6)*+{A}="0";
(0,6)*+{B}="2";
(12,6)*+{C}="4";
(-12,-6)*+{A^{\prime}}="10";
(0,-6)*+{B^{\prime}}="12";
(12,-6)*+{C^{\prime}}="14";
{\ar^{x} "0";"2"};
{\ar^{y} "2";"4"};
{\ar_{a} "0";"10"};
{\ar^{b} "2";"12"};
{\ar^{c} "4";"14"};
{\ar^{x^{\prime}} "10";"12"};
{\ar^{y^{\prime}} "12";"14"};
{\ar@{}|{} "0";"12"};
{\ar@{}|{} "2";"14"};
\endxy
$$
\end{itemize}
In the above situation, we say that the triplet $(a,b,c)$ {\it realizes} $(a,c)$.
\end{defn}

\begin{defn}
Let $\B,\mathbb{E}$ be as above. A realization of $\mathbb{E}$ is said to be {\it additive}, if it satisfies the following conditions.
\begin{itemize}
\item[{\rm (i)}] For any $A,C\in\B$, the split $\mathbb{E}$-extension $0\in\mathbb{E}(C,A)$ satisfies
\[ \mathfrak{s}(0)=0. \]
\item[{\rm (ii)}] For any pair of $\mathbb{E}$-extensions $\delta\in\mathbb{E}(C,A)$ and $\delta^{\prime}\in\mathbb{E}(C^{\prime},A^{\prime})$, we have
\[ \mathfrak{s}(\delta\oplus\delta^{\prime})=\mathfrak{s}(\delta)\oplus\mathfrak{s}(\delta^{\prime}). \]
\end{itemize}
\end{defn}

\begin{defn}\cite[Definition 2.12]{NP}
A triplet $(\B,\mathbb{E},\mathfrak{s})$ is called an {\it extriangulated category} if it satisfies the following conditions.
\begin{itemize}
\item[{\rm (ET1)}] $\mathbb{E}\colon\B^{\mathrm{op}}\times\B\to\Ab$ is an additive bifunctor.
\item[{\rm (ET2)}] $\mathfrak{s}$ is an additive realization of $\mathbb{E}$.
\item[{\rm (ET3)}] Let $\delta\in\mathbb{E}(C,A)$ and $\delta^{\prime}\in\mathbb{E}(C^{\prime},A^{\prime})$ be any pair of $\mathbb{E}$-extensions, realized as
\[ \mathfrak{s}(\delta)=[A\overset{x}{\longrightarrow}B\overset{y}{\longrightarrow}C],\ \ \mathfrak{s}(\delta^{\prime})=[A^{\prime}\overset{x^{\prime}}{\longrightarrow}B^{\prime}\overset{y^{\prime}}{\longrightarrow}C^{\prime}]. \]
For any commutative square
$$
\xy
(-12,6)*+{A}="0";
(0,6)*+{B}="2";
(12,6)*+{C}="4";
(-12,-6)*+{A^{\prime}}="10";
(0,-6)*+{B^{\prime}}="12";
(12,-6)*+{C^{\prime}}="14";
{\ar^{x} "0";"2"};
{\ar^{y} "2";"4"};
{\ar_{a} "0";"10"};
{\ar^{b} "2";"12"};
{\ar^{x^{\prime}} "10";"12"};
{\ar^{y^{\prime}} "12";"14"};
{\ar@{}|{} "0";"12"};
\endxy
$$
in $\B$, there exists a morphism $(a,c)\colon\delta\to\delta^{\prime}$ satisfying $cy=y^{\prime}b$.
\item[{\rm (ET3)$^{\mathrm{op}}$}] Dual of {\rm (ET3)}.
\item[{\rm (ET4)}] Let $\delta\in\mathbb{E}(D,A)$ and $\delta^{\prime}\in\mathbb{E}(F,B)$ be $\mathbb{E}$-extensions realized by
\[ A\overset{f}{\longrightarrow}B\overset{f^{\prime}}{\longrightarrow}D\ \ \text{and}\ \ B\overset{g}{\longrightarrow}C\overset{g^{\prime}}{\longrightarrow}F \]
respectively. Then there exist an object $E\in\B$, a commutative diagram
$$
\xy
(-21,7)*+{A}="0";
(-7,7)*+{B}="2";
(7,7)*+{D}="4";
(-21,-7)*+{A}="10";
(-7,-7)*+{C}="12";
(7,-7)*+{E}="14";
(-7,-21)*+{F}="22";
(7,-21)*+{F}="24";
{\ar^{f} "0";"2"};
{\ar^{f^{\prime}} "2";"4"};
{\ar@{=} "0";"10"};
{\ar_{g} "2";"12"};
{\ar^{d} "4";"14"};
{\ar^{h} "10";"12"};
{\ar^{h^{\prime}} "12";"14"};
{\ar_{g^{\prime}} "12";"22"};
{\ar^{e} "14";"24"};
{\ar@{=} "22";"24"};
{\ar@{}|{} "0";"12"};
{\ar@{}|{} "2";"14"};
{\ar@{}|{} "12";"24"};
\endxy
$$
in $\B$, and an $\mathbb{E}$-extension $\delta^{\prime\prime}\in\mathbb{E}(E,A)$ realized by $A\overset{h}{\longrightarrow}C\overset{h^{\prime}}{\longrightarrow}E$, which satisfy the following compatibilities.
\begin{itemize}
\item[{\rm (i)}] $D\overset{d}{\longrightarrow}E\overset{e}{\longrightarrow}F$ realizes $f^{\prime}_{\ast}\delta^{\prime}$,
\item[{\rm (ii)}] $d^{\ast}\delta^{\prime\prime}=\delta$,

\item[{\rm (iii)}] $f_{\ast}\delta^{\prime\prime}=e^{\ast}\delta^{\prime}$.
\end{itemize}

\item[{\rm (ET4)$^{\mathrm{op}}$}] Dual of {\rm (ET4)}.
\end{itemize}
\end{defn}

\begin{rem}
Note that both exact categories and triangulated categories are extriangulated categories, see \cite[Example 2.13]{NP} and extension-closed subcategories of extriangulated categories are
again extriangulated, see \cite[Remark 2.18]{NP} . Moreover, there exist extriangulated categories which
are neither exact categories nor triangulated categories, see \cite[Proposition 3.30]{NP} and \cite[Example 4.14]{ZZ}.
\end{rem}

We will use the following terminology.
\begin{defn}{\cite{NP}}
Let $(\B,\EE,\mathfrak{s})$ be an extriangulated category.
\begin{itemize}
\item[(1)] A sequence $A\xrightarrow{~x~}B\xrightarrow{~y~}C$ is called a {\it conflation} if it realizes some $\EE$-extension $\del\in\EE(C,A)$. In this case, $x$ is called an {\it inflation} and $y$ is called a {\it deflation}.

\item[(2)] If a conflation  $A\xrightarrow{~x~}B\xrightarrow{~y~}C$ realizes $\delta\in\mathbb{E}(C,A)$, we call the pair $( A\xrightarrow{~x~}B\xrightarrow{~y~}C,\delta)$ an {\it $\EE$-triangle}, and write it in the following way.
$$A\overset{x}{\longrightarrow}B\overset{y}{\longrightarrow}C\overset{\delta}{\dashrightarrow}$$
We usually do not write this $``\delta"$ if it is not used in the argument.
\item[(3)] Let $A\overset{x}{\longrightarrow}B\overset{y}{\longrightarrow}C\overset{\delta}{\dashrightarrow}$ and $A^{\prime}\overset{x^{\prime}}{\longrightarrow}B^{\prime}\overset{y^{\prime}}{\longrightarrow}C^{\prime}\overset{\delta^{\prime}}{\dashrightarrow}$ be any pair of $\EE$-triangles. If a triplet $(a,b,c)$ realizes $(a,c)\colon\delta\to\delta^{\prime}$, then we write it as
$$\xymatrix{
A \ar[r]^x \ar[d]^a & B\ar[r]^y \ar[d]^{b} & C\ar@{-->}[r]^{\del}\ar[d]^c&\\
A'\ar[r]^{x'} & B' \ar[r]^{y'} & C'\ar@{-->}[r]^{\del'} &}$$
and call $(a,b,c)$ a {\it morphism of $\EE$-triangles}.

\item[(4)] An object $P\in\B$ is called {\it projective} if
for any $\EE$-triangle $A\overset{x}{\longrightarrow}B\overset{y}{\longrightarrow}C\overset{\delta}{\dashrightarrow}$ and any morphism $c\in\B(P,C)$, there exists $b\in\B(P,B)$ satisfying $yb=c$.
We denote the subcategory of projective objects by $\mathcal P\subseteq\B$. Dually, the subcategory of injective objects is denoted by $\I\subseteq\B$.

\item[(5)] We say that $\B$ {\it has enough projective objects} if
for any object $C\in\B$, there exists an $\EE$-triangle
$A\overset{x}{\longrightarrow}P\overset{y}{\longrightarrow}C\overset{\delta}{\dashrightarrow}$
satisfying $P\in\mathcal P$. Dually we can define $\B$ {\it has enough injective objects}.

\item[(6)] Let $\mathcal{X}$ be a subcategory of $\B$. We say $\mathcal{X}$ is {\it extension-closed}
if a conflation $A\xrightarrow{~~}B\xrightarrow{~~}C$satisfies $A,C\in\mathcal{X}$, then $B\in\mathcal{X}$.
\end{itemize}
\end{defn}

In this article, we always assume $\B$ has enough projectives and enough injectives.
\medskip

By \cite{NP}, we give the following useful remark, which will be used later in the proofs.

\begin{rem}\label{useful}
Let $\xymatrix{A\ar[r]^a &B \ar[r]^b &C \ar@{-->}[r] &}$ and $\xymatrix{X\ar[r]^x &Y \ar[r]^y &Z \ar@{-->}[r] &}$ be two $\EE$-triangles. Then
\begin{itemize}
\item[(a)] In this following commutative diagram
$$\xymatrix{
X\ar[r]^x \ar[d]_f &Y \ar[d]^g \ar[r]^y &Z \ar[d]^h \ar@{-->}[r] &\\
A\ar[r]^a &B \ar[r]^b &C \ar@{-->}[r] &}
$$
$f$ factors through $x$ if and only if $h$ factors through $b$.
\item[(b)] In the following commutative diagram
$$\xymatrix{
A\ar[r]^a \ar[d]_s &B \ar[d]^r \ar[r]^b &C \ar[d]^t \ar@{-->}[r] &\\
X\ar[r]^x \ar[d]_f &Y \ar[d]^g \ar[r]^y &Z \ar[d]^h \ar@{-->}[r] &\\
A\ar[r]^a &B \ar[r]^b &C \ar@{-->}[r] &}
$$
$fs=1_A$ implies $B$ is a direct summand of $C\oplus Y$ and $C$ is a direct summand of $Z\oplus B$; $ht=1_C$ implies $B$ is a direct summand of $A\oplus Y$ and $A$ is a direct summand of $X\oplus B$.
\item[(c)] If we have $b:B\xrightarrow{d_1} D \xrightarrow{d_2} C$ and $d_2: D\xrightarrow{d_3} B\xrightarrow{b} C$, then $B$ is a direct summand of $A\oplus D$.\\
If we have $a: A\xrightarrow{e_1} E\xrightarrow{e_2} B$ and $e_1:A \xrightarrow{a} B\xrightarrow{e_3} E$, then $B$ is a direct summand of $C\oplus E$.
\item[(d)] $\Hom_{\underline \B}(\Omega \T,X)=0$ if and only if $\EE(\T,X)=0$, $\Hom_{\overline \B}(Y,\Sigma \T)=0$ if and only if $\EE(Y,\T)=0$.
\end{itemize}
\end{rem}

We recall the following proposition (\cite[Proposition 1.20]{LN}), which (also the dual of it) will be used many times in the article.

\begin{prop}
Let $A\overset{x}{\longrightarrow}B\overset{y}{\longrightarrow}C\overset{\delta}{\dashrightarrow}$ be any $\EE$-triangle, let $f\colon A\rightarrow D$ be any morphism, and let $D\overset{d}{\longrightarrow}E\overset{e}{\longrightarrow}C\overset{f_{\ast}\delta}{\dashrightarrow}$ be any $\EE$-triangle realizing $f_{\ast}\delta$. Then there is a morphism $g$ which gives a morphism of $\EE$-triangles
$$\xymatrix{
A \ar[r]^{x} \ar[d]_f &B \ar[r]^{y} \ar[d]^g &C \ar@{=}[d]\ar@{-->}[r]^{\delta}&\\
D \ar[r]_{d} &E \ar[r]_{e} &C\ar@{-->}[r]_{f_{\ast}\delta}&
}
$$
and moreover, the sequence $A\xrightarrow{\svecv{f}{x}}D\oplus B\xrightarrow{\svech{d}{-g}}E\overset{e^{\ast}\delta}{\dashrightarrow}$ becomes an $\EE$-triangle.
\end{prop}

\section{The First Main Result}

From now on, for convenience, we denote $\Hom_{\underline \B}(\Omega \T,-)$ by $H$ and subcategory $\{Y \text{ }| \text{ } \Hom_{\underline \B}(\Omega \T,Y)= 0 \}$ by $\K$.
In this article, we assume that $\T$ is a contravariantly finite subcategory of $\B$ such that $\mathcal P\subset \T$ and $\EE(\T,\mathcal P)=0$.

\begin{lem}\label{contra}
The subcategory $\Omega \T$ is contravariantly finite in $\B$.
\end{lem}

\begin{proof}
An object $B$ admits an $\EE$-triangle $\xymatrix{B\ar[r] &I \ar[r] &\Sigma B \ar@{-->}[r] &}$ where $I\in \mathcal I$. Since $\T$ is contravariantly finite, we have a right $\T$-approximation $T_0\xrightarrow{~f~} \Sigma B$ of $\Sigma B$. An object $T$ admits an $\EE$-triangle $$\xymatrix{\Omega T \ar[r] &P \ar[r] &T \ar@{-->}[r] &}$$ where $P\in \mathcal P$. Then we have the following commutative diagram
$$\xymatrix{
\Omega T \ar[d]_{g} \ar[r]^p &P \ar[d] \ar[r] &T \ar[d]^f \ar@{-->}[r] &\\
B\ar[r] &I \ar[r] &\Sigma B \ar@{-->}[r] &
}
$$
Let $\Omega T_0\in \Omega \T$ which admits an $\EE$-triangle $\xymatrix{\Omega T_0 \ar[r] &P_0 \ar[r] &T_0 \ar@{-->}[r] &}$ where $P_0\in \mathcal P$ and $T_0\in \mathcal \T$, for any morphism $g_0:\Omega T_0\to B$, we have the following commutative diagram
$$\xymatrix{
\Omega T_0 \ar[d]_{g_0} \ar[r] &P_0 \ar[d] \ar[r] &T_0 \ar[d]^{f_0} \ar@{-->}[r] &\\
B\ar[r] &I \ar[r] &\Sigma B \ar@{-->}[r] &
}
$$
Since $f$ is a right $\T$-approximation of $\Sigma B$, there is a morphism $t:T_0\to T$ such that $f_0=ft$. Hence we have the following commutative diagram
$$\xymatrix{
\Omega T_0 \ar[d]_{t'} \ar[r] &P_0 \ar[d] \ar[r] &T_0 \ar[d]^{t} \ar@{-->}[r] &\\
\Omega T \ar[r] &P \ar[r] &T \ar@{-->}[r] &
}
$$
This implies $g_0-gt'$ factors through $p$, which means $\Omega T\oplus P\xrightarrow{\svech{g}{p}} B$ is a right $\Omega \T$-approximation of $B$. Hence $\Omega \T$ is contravariantly finite in $\B$.
\end{proof}

\begin{prop}\label{abelian}
The category $\mod\underline {\Omega \T}$ is abelian.
\end{prop}

\begin{proof}
It is enough to show that $\underline {\Omega \T}$ has pseudo-kernels. For a morphism $\underline f: \Omega T_1\to \Omega T_0$, we have the following commutative diagram
$$\xymatrix{
X\ar[r] \ar@{=}[d] & Y \ar[r]^g \ar[d]^q &\Omega T_1 \ar[d]^f \ar@{-->}[r] &\\
X \ar[r] &P_0\ar[r]^p &\Omega T_0 \ar@{-->}[r] &
}
$$
where $P_0\in \mathcal P$. Then we get an $\EE$-triangle $\xymatrix{Y \ar[r]^-{\svecv{g}{-q}} &\Omega T_1\oplus P_0 \ar[r]^-{\svech{f}{p}} &\Omega T_0 \ar@{-->}[r] &}$. Let $h: \Omega T_2\to Y$ be a right $\Omega \T$-approximation of $Y$, we claim that $\underline {gh}$ is a pseudo-kernel of $\underline f$.

Let $\underline t: \Omega T\to \Omega T_1$ be a morphism such that $\underline {tf}=0$. Then we have a commutative diagram
$$\xymatrix{
&\Omega T \ar[r]^{p_1} \ar[d]^{\svecv{t}{0}} &P \ar[d]^{p_2} \\
Y \ar[r]_-{\svecv{g}{-q}} &\Omega T_1\oplus P_0 \ar[r]_-{\svech{f}{p}} &\Omega T_0 \ar@{-->}[r] &
}
$$
where $P\in \mathcal P$. Since $P$ is projective, $p_2$ factors through $\svech{f}{p}$, there is a morphism $p_3:P \to \Omega T_1$ such that $f(t-p_3p_1)=0$. This implies $t-p_3p_1$ factors through $g$, there is a morphism $t':\Omega T\to Y$ such that $gt'=t-p_3p_1$. Since $h$ is a right $\Omega \T$-approximation of $Y$, $t'$ factors through $h$. Hence $\underline t$ factors through $\underline {gh}$.
\end{proof}

\begin{lem}\label{summand}
The subcategory $\Omega \T$ is closed under direct summands.
\end{lem}

\begin{proof}
Let $\xymatrix{X\oplus Y \ar[r]^-{\svech{x}{y}} &P \ar[r] &T \ar@{-->}[r] &}$ be an $\EE$-triangle where $T\in \T$ and $P\in \mathcal P$, then $x$ is an inflation and it admits an $\EE$-triangle $\xymatrix{X \ar[r]^x &P \ar[r]  &T_X \ar@{-->}[r]&}$. Since $\EE(T,P)=0$, then there exists a morphism $f:T\to T$ such that $fx=x$ and $fy=0$. Hence we have the following commutative diagram
$$\xymatrix{
X \ar[r]^-{x} \ar[d]_-{\svecv{1}{0}} &P \ar@{=}[d] \ar[r]^c &T_X \ar[d]^a \ar@{-->}[r] &\\
X\oplus Y \ar[r]^-{\svech{x}{y}} \ar[d]_-{\svech{1}{0}} &P \ar[d]^f \ar[r] &T \ar[d]^b \ar@{-->}[r] &\\
X \ar[r]^-{x} &P \ar[r]^c &T_X \ar@{-->}[r] &.
}
$$
Then there exists a morphism $d:T_X\rightarrow P$ such that $1_{T_X}-ba=cd$, hence $T_X$ is a direct summand of $T\oplus P$. Since $\T$ is closed under direct summands, we get $T_X\in \T$ and then $X\in \Omega \T$.
\end{proof}

\begin{lem}\label{min}
Let $f: R_1\to R_2$ be a morphism in $\Omega \T$, then $H(\underline f)$ is right minimal if and only if $\underline f$ is right minimal.
\end{lem}

\begin{proof}
Let $\alpha: H(R_1)\to H(R_1)$ be a morphism in $\mod\underline {\Omega \T}$, then by Yoneda's Lemma, there is a morphism $r:R_1\to R_1$ such that $H(\underline r)=\alpha$.
\begin{itemize}
\item When $ H(\underline f)\alpha =H(\underline f)$, we have $\underline {fr}=\underline f$, hence if $\underline f$ is right minimal, we get an isomorphism $\underline r$ , then $\alpha$ is also an isomorphism. This means $H(\underline f)$ is right minimal.
\item For any morphism $r:R_1\to R_1$ such that $\underline {fr}=\underline f$, we get $H(\underline f)H(\underline r)=H(\underline f)$, when $H(\underline f)$ is right minimal, $H(\underline r)$ is an isomorphism, which means $\underline r$ is an isomorphism and $\underline f$ is right minimal.

\end{itemize}
\end{proof}

Let $\widehat{\T}$ be the subcategory of non-projective objects in $\T$.

\begin{lem}\label{tec}
Let $f:R_1\to R_2$ be a morphism in $\Omega \T$, $R_i$ admits an $\EE$-triangle $\xymatrix{R_i \ar[r]^{p_i} &P_i \ar[r]^{q_i} &T_i \ar@{-->}[r] &}$ where $P_i\in \mathcal P$, $T_i\in \widehat {\T}$, then we have the following commutative diagram
$$\xymatrix{
R_1 \ar@{=}[r] \ar[d]_{f'} &R_1\ar[d]\\
R_2' \ar[d]_g \ar[r] &P \ar[r] \ar[d] &T_2 \ar@{=}[d] \ar@{-->}[r]&\\
B \ar[r]_h \ar@{-->}[d] &T_1' \ar[r] \ar@{-->}[d] &T_2 \ar@{-->}[r] &\\
&& &&
}
$$
where $P\in \mathcal P$, $\underline {f'}=\underline f$, $R'_2=R_2$ and  $T_1'\simeq T_1$ in $\underline \B$.
\end{lem}

\begin{proof}
First we have the following commutative diagram
$$\xymatrix{
R_1 \ar[r]^{p_1} \ar[d]_f &P_1 \ar[d] \ar[r] &T_1 \ar@{=}[d] \ar@{-->}[r] &\\
R_2 \ar[r] &B \ar[r] &T_1 \ar@{-->}[r] &
}
$$
which induces an $\EE$-triangle $\xymatrix{R_1 \ar[r]^-{\svecv{f}{p_1}} &R_2\oplus P_1 \ar[r] &B \ar@{-->}[r] &}$. Then we have the following commutative diagram
$$\xymatrix{
R_1 \ar@{=}[r] \ar[d]_-{\svecv{f}{p_1}}  &R_1\ar[d]^-{\svecv{p_2f}{p_1}} \\
R_2\oplus P_1 \ar[d] \ar[r]^-{\left(\begin{smallmatrix}
p_2&0\\
0&1
\end{smallmatrix}\right)} &P_2\oplus P_1 \ar[r]_-{\svech{q_2}{0}} \ar[d] &T_2 \ar@{=}[d] \ar@{-->}[r]&\\
B \ar[r]_h \ar@{-->}[d] &T_1' \ar[r] \ar@{-->}[d] &T_2 \ar@{-->}[r] &\\
&& &&
}
$$
We write $R_1 \xrightarrow{\svecv{f}{p_1}} R_2\oplus P_1$ as $R_1 \xrightarrow{f'} R_2'$, then $\underline {f'}=\underline f$ and $R_2'=R_2$ in $\underline B$. Since $\EE(\T,\mathcal P)=0$, we also have the following commutative diagram
$$\xymatrix{
R_1 \ar[r]^-{\svecv{p_2f}{p_1}} \ar@{=}[d] &P_2 \oplus P_1 \ar[r] \ar[d]^-{\svech{0}{1}} &T_1' \ar[d]^t \ar@{-->}[r] &\\
R_1 \ar[r]^{p_1} \ar@{=}[d] &P_1 \ar[d]^p \ar[r]^{q_1} &T_1 \ar[d]^{t'} \ar@{-->}[r] &\\
R_1 \ar[r]^-{\svecv{p_2f}{p_1}} &P_2 \oplus P_1 \ar[r]  &T_1' \ar@{-->}[r] &
}
$$
Then $T_1'$ is a direct summand of $T_1\oplus P_2\oplus P_1$. If $T_1'\in \mathcal P$, we have $R_1\in \mathcal P$ and $T_1\in \mathcal P$, but $T_1\in \widehat{\T}$, a contradiction. Hence $T_1'$ is a direct summand of $T_1$ in $\underline \B$. On the other hand, by the same method we get that $T_1$ is a direct summand of $T_1'$. Hence $T_1'\simeq T_1$ in $\underline \B$.
\end{proof}

\begin{lem}\label{min2}
If we have the following commutative diagram
$$\xymatrix{
\Omega B \ar[r] \ar[d]_{f'} &P_B \ar[d] \ar[r] &B \ar[d]^f \ar@{-->}[r] &\\
\Omega A \ar[r]  &P_A  \ar[r] &A \ar@{-->}[r] &
}
$$
where $A,B\in \T$ and $P_A,P_B\in \mathcal P$, then $\underline {f'}$ is right minimal implies $\underline f$ is also right minimal.
\end{lem}

\begin{proof}
Let $\underline b:B\to B$ be a morphism such that $\underline {fb}=\underline f$, we have the following commutative diagram
$$\xymatrix{
\Omega B \ar[r] \ar[d]_{b'} &P_B \ar[d] \ar[r] &B \ar[d]^b \ar@{-->}[r] &\\
\Omega B \ar[r] \ar[d]_{f'} &P_B \ar[d] \ar[r] &B \ar[d]^f \ar@{-->}[r] &\\
\Omega A \ar[r]  &P_A  \ar[r] &A \ar@{-->}[r] &
}
$$
Then we have
$$\xymatrix{
\Omega B \ar[r]^{p'} \ar[d]_{f'b'-f'} &P_B \ar[d] \ar[r] &B \ar[d]^{fb-f} \ar@{-->}[r] &\\
\Omega A \ar[r]  &P_A  \ar[r]^p &A \ar@{-->}[r] &
}
$$
Since $\underline {fb}=\underline f$, $fb-f$ factors through $\mathcal P$, then it factors through $p$, hence $f'b'-f'$ factors through $p'$. Then $\underline {f'b'}=\underline {f'}$. Since $\underline {f'}$ is right minimal, $\underline {b'}$ is an isomorphism. Let $\underline {c'}$ be the inverse of $b'$, then we have the following commutative diagram
$$\xymatrix{
\Omega B \ar[r] \ar[d]_{c'} &P_B \ar[d] \ar[r] &B \ar[d]^c \ar@{-->}[r] &\\
\Omega B \ar[r]  &P_B  \ar[r] &B \ar@{-->}[r] &
}
$$
We can easily check that $\underline c$ is the inverse of $\underline b$, the proof is left to the readers. This means $\underline b$ is an isomorphism, then $\underline f$ is right minimal.
\end{proof}

%For convenience, we denote $\Hom_{\underline \B}(\Omega \T,-)$ by $H$.

\begin{lem}\label{exact}
If we have the following commutative diagram
$$\xymatrix{
R_1 \ar@{=}[r] \ar[d]_f &R_1 \ar[d]\\
R_2 \ar[d]_g \ar[r]^q &P \ar[r] \ar[d]^p &T_2 \ar@{=}[d] \ar@{-->}[r]&\\
A \ar[r]_h \ar@{-->}[d] &T_1 \ar[r] \ar@{-->}[d] &T_2 \ar@{-->}[r] &\\
&& &&
}
$$
where $T_1,T_2\in \T$, $P\in \mathcal P$, then by applying $H$, we can get the following exact sequence in $\mod\underline {\Omega \T}$:
$$H(R_1) \xrightarrow{H(\underline f)} H(R_2) \xrightarrow{H(\underline g)}  H(A) \xrightarrow{H(\underline h)} H(T_1)$$
\end{lem}

\begin{proof}
It is obvious that $H(\underline g)H(\underline f)=0$. Let $R\in \Omega \T$ and $r:R\to R_2$ be a morphism $\underline {gr}=0$, then as in the proof of Proposition \ref{abelian}, $\underline r$ factors through $\underline f$. Since we have an $\EE$-triangle $$\xymatrix{R_2 \ar[r]^-{\svecv{g}{-q}} &A\oplus P \ar[r]^-{\svech{h}{p}} &T_1 \ar@{-->}[r] &},$$ hence $H(\underline h)H(\underline g)=0$ and by the similar argument as in the proof of Proposition \ref{abelian}, any morphism $r':R\to A$ such that $\underline {hr'}=0$ factors through $\underline g$. Hence we get the required exact sequence.
\end{proof}

\begin{lem}
For any objects $B,C$, let $\mathcal J (B,C)=\{g\in \Hom_{\B}(B,C) \text{ }|\text{ } H(\underline f) =0\}$, then $\mathcal J$ is an ideal of $\B$.
\end{lem}

\begin{proof}
We only need to check that if $g\in \J(B,C)$, then for any morphism $f:A\to B$ and $h:B\to C$, we have $hgf\in \J(A,C)$.\\
By Lemma \ref{contra}, there is a right $\Omega \T$-approximation $r_A:R_A\to A$ of $A$, then $H(\underline {hgf})=0$ if and only if $\underline {hgfr_A}=0$. There is also a right $\Omega \T$-approximation $r_B:R_B\to B$ of $B$, hence there exists a morphism $r:R_A\to R_B$ such that $r_Br=fr_A$. Since $H(\underline {g})=0$ if and only if $\underline {gr_B}=0$, we have $H(\underline {g})=0 \Leftrightarrow \underline {gr_B}=0  \Rightarrow \underline {hgr_Br}=0 \Leftrightarrow \underline {hgfr_A}=0\Leftrightarrow H(\underline {hgf})=0$.
\end{proof}

\begin{rem}\label{ideal}
We get following commutative diagram if $H$ is a quotient functor
$$\xymatrix@C=0.5cm@R0.5cm{
\B \ar[rr]^H \ar[dr]_{\pi} && \mod \underline{\Omega \T}\\
&\B/\mathcal J \ar[ur]_-{\simeq}
}
$$
where $\pi:\B\to \B/\mathcal J$ is the natural quotient functor. Since $\mod \underline{\Omega \T}$ is abelian, $\B/\mathcal J $ is also abelian.
\end{rem}

The following remark is very useful in the proofs.

\begin{rem}\label{remark}
If we have a commutative diagram
$$\xymatrix{
\Omega B \ar@{=}[r] \ar[d]_f &\Omega B \ar[d]\\
\Omega C \ar[d] \ar[r] &P \ar[r] \ar[d] &C \ar@{=}[d] \ar@{-->}[r]&\\
A \ar[r] \ar@{-->}[d] &B \ar[r] \ar@{-->}[d] &C \ar@{-->}[r] &\\
&& &&
}
$$
where $P\in \mathcal P$ and $B\in \T$, then $\Hom_{\B}(f,P_0)$ is full if $P_0\in \mathcal P$.
\end{rem}

\textbf{Now we can prove Theorem \ref{main1}.}

\begin{proof}
(I) We show (a) implies dense.

Let $0\neq F\in \mod\underline {\Omega \T}$. Then it admits the following exact sequence: $H(R_1) \xrightarrow{H(\underline f)} H(R_2) \xrightarrow{\alpha} F\to 0$ where $R_1,R_2\in \Omega \T$. If $\alpha$ is already an isomorphism, then there is nothing more to show. Otherwise $H(\underline f)$ can be the composition of the monomorphism $\Ker \alpha \to H(R_2)$ and a right minimal epimorphism $H(R_1)\to \Ker \alpha$, hence $H(\underline f)$ is right minimal, now by Lemma \ref{min}, $\underline f$ is right minimal. Under this setting, we have $R_i\in \Omega \widehat{\T}$. By Lemma \ref{tec}, there exists a commutative diagram
$$\xymatrix{
R_1 \ar@{=}[r] \ar[d]_{f'} &R_1\ar[d]\\
R_2' \ar[d]_g \ar[r] &P \ar[r] \ar[d] &T_2 \ar@{=}[d] \ar@{-->}[r]&\\
B \ar[r]_h \ar@{-->}[d] &T_1' \ar[r] \ar@{-->}[d] &T_2 \ar@{-->}[r] &\\
&& &&
}
$$
where $P\in \mathcal P$, $\underline {f'}=\underline f$, $R'_2=R_2$ and  $T_i\in \T$ in $\underline B$. By (a), $H(\underline h)=0$. By Lemma \ref{exact}, we get an exact sequence $$H(R_1) \xrightarrow{H(\underline f)} H(R_2) \xrightarrow{H(\underline g)} H(B)\to 0.$$ Hence there is a commutative diagram:
$$\xymatrix{
H(R_1) \ar@{=}[d] \ar[r]^-{H(\underline f)} &H(R_2) \ar@{=}[d] \ar[r]^-{\alpha} &F \ar@{.>}[d]^{\beta} \ar[r] &0\\
H(R_1) \ar[r]^-{H(\underline f)} &H(R_2) \ar[r]^-{H(\underline g)} &H(B)
\ar[r] &0
}
$$
Since $\mod\underline {\Omega \T}$ is abelian, $\alpha$ and $H(\underline g)$ are cokernels of $H(\underline f)$, we get $\beta$ is an isomorphism.
\medskip

(II) We show (b) implies full.
\medskip

Let $A,A'$ be objects in $\B$, we assume they do not have direct summand in $\K$. By (b), we have the following commutative diagrams
%$$\xymatrix{
%R_1 \ar@{=}[r] \ar[d]_f &R_1 \ar[d]\\
%R_2 \ar[d]_g \ar[r] &P \ar[r] \ar[d] &T^2 \ar@{=}[d] \ar@{-->}[r]&\\
%A \ar[r]_h  &T^1 \ar[r]  &T^2 \ar@{-->}[r] &\\
%}\quad
%\xymatrix{
%R_1' \ar@{=}[r] \ar[d]_{f'} &R_1' \ar[d]\\
%R_2' \ar[d]_{g'} \ar[r] &P' \ar[r] \ar[d] &T^2' \ar@{=}[d] \ar@{-->}[r]&\\
%A' \ar[r]_{h'}  &T^1' \ar[r]  &T^2' \ar@{-->}[r] &
%}
%$$
$$\xymatrix{
R_1 \ar@{=}[r] \ar[d]_f &R_1 \ar[d]^p\\
R_2 \ar[d]_g \ar[r]_r &P \ar[r] \ar[d] &T_2 \ar@{=}[d] \ar@{-->}[r]&\\
A \ar[r]_h \ar@{-->}[d] &T_1 \ar[r] \ar@{-->}[d] &T_2 \ar@{-->}[r] &\\
&& &&
\\} \quad
\xymatrix{
R_1' \ar@{=}[r] \ar[d]_{f'} &R_1' \ar[d]\\
R_2' \ar[d]_{g'} \ar[r] &P' \ar[r] \ar[d] &{T_2}' \ar@{=}[d] \ar@{-->}[r]&\\
A' \ar[r]_{h'} \ar@{-->}[d] &{T_1}' \ar[r] \ar@{-->}[d] &{T_2}' \ar@{-->}[r] &\\
&& &&
}$$

Let $\alpha: H(A)\to H(A')$ be a morphism in $\mod\underline {\Omega \T}$. Since $H(\underline {\Omega \T})$ is the subcategory of projective objects in $\mod\underline {\Omega \T}$, by Lemma \ref{exact} we have the following commutative diagram of exact sequences
$$\xymatrix{
H(R_1) \ar@{.>}[d]^{\alpha_1} \ar[r]^-{H(\underline f)} &H(R_2) \ar@{.>}[d]_{\alpha_2} \ar[r]^{H(\underline g)} &H(A) \ar[d]^{\alpha} \ar[r] &0\\
H(R_1') \ar[r]^-{H(\underline {f'})} &H(R_2') \ar[r]^-{H(\underline {g'})} &H(A')  \ar[r] &0
}
$$
By Yoneda's Lemma, $\alpha_i=H(\underline {r_i})$, hence $\underline {r_2f}=\underline {f'r_1}$, then we have $g'r_2f:R_1\xrightarrow{p} P_0\xrightarrow{q_0} A'$ where $P_0\in \mathcal P$. By Remark \ref{remark}, we have $p:R_1\xrightarrow{f} R_2\xrightarrow{r'} P_0$. Since $P_0$ is a projective object, we have $q_0:P_0 \xrightarrow{q'} R_2'\xrightarrow{g'} A'$. Hence $g'r_2f=g'q'p$, then we have $r_2f-q'p=(r_2-q'r')f: R_1\xrightarrow{r_1'} R_1'\xrightarrow{f'} R_2'$.  Hence we have the following commutative diagram:
$$\xymatrix{
R_1 \ar[r]^f \ar[d]_{r_1'} &R_2 \ar[r]^g \ar[d]^{r_2-q'r'} &A \ar[d]^a \ar@{-->}[r] &\\
R_1' \ar[r]_{f'} &R_2' \ar[r]_{g'} &A' \ar@{-->}[r] &
}
$$
which implies a commutative diagram
$$
\xymatrix{
H(R_2) \ar[r]^-{H(\underline g)} \ar[d]_{H(\underline {r_2})} &H(A) \ar[d]^{H(\underline a)} \ar[r] &0\\
H(R_2') \ar[r]_-{H(\underline {g'})} &H(A') \ar[r] &0
}
$$
Hence $\alpha=H(\underline a)$.
\medskip

(III) We show full and dense implies (a).

\medskip

Assume in the following commutative diagram
$$\xymatrix{
R_1 \ar@{=}[r] \ar[d]_{f} &R_1\ar[d]^{p_1}\\
R_2 \ar[d]_g \ar[r]^{r} &P \ar[r] \ar[d] &T_2 \ar@{=}[d] \ar@{-->}[r]&\\
A \ar[r]_h \ar@{-->}[d] &T_1 \ar[r]_t \ar@{-->}[d] &T_2 \ar@{-->}[r] &\\
&& &&
}
$$
where $P\in \mathcal P$ and $T_i\in \T$, $\underline f\neq 0$ is right minimal, we show that $H(\underline h)$. By Lemma \ref{exact}, We have the following exact sequence:
$$H(R_1) \xrightarrow{H(\underline f)} H(R_2) \xrightarrow{H(\underline g)}  H(B) \xrightarrow{H(\underline h)} H(T_1)$$
If $H(\underline g)=0$, we get $\underline g=0$, which implies $g$ factors through $r'$, hence $\underline t$ is a split epimorphism, a contradiction to the fact that $\underline t$ is right minimal by Lemma \ref{min}. Since $H$ is full and dense, we have epic-monic factorization $H(R_2) \xrightarrow{H(\underline u)} H(C) \xrightarrow{H(\underline v)} H(B)$ of $H(\underline g)$, then $\underline g=\underline {vu}$. We can assume that $C$ does not have direct summand in $\K$. Since $H(\underline u)H(\underline f)=0$, we have $\underline {uf}=0$, then it has the form $uf:R_1 \xrightarrow{p_1} P \xrightarrow{q_1} C$. Hence $(u-q_1r)f=0$, and there is a morphism $v':B\to C$ such that $u-q_1r=v'g$. Denote $u-q_1r$ by $u'$, we have $\underline {u'}=\underline u$. Since $g-vu'$ factors through $r$, it has the form $R_2\xrightarrow{r} P \xrightarrow{q_B} B$. Now $\underline {v'vu'}=\underline {v'g}=\underline {u'}$, we will show that $\underline u'$ is left minimal.

Since $\B$ is Krull-Schmidt, $\B/\mathcal P$ is also Krull-Schmidt. Then $\underline{u'}$ has the form $R_2\xrightarrow{\svecv{\underline{u_1}}{0}} C_1\oplus C_2$ where $\underline{u_1}$ is a left minimal morphism. Then we have $$H(\underline {u'}): H(R_2)\xrightarrow{\svecv{H(\underline{u_1})}{0}} H(C_1)\oplus H(C_2).$$ Since $H(\underline {u'})$ is an epimorphism, we have $H(C_2)=0$, which implies $C_2=0$ by the assumption of $C$. Hence $\underline {u'}$ is also left minimal. Then $\underline {v'v}=\underline {1_{C}}$, $C$ is a direct summand of $A\oplus P'$ where $P'\in \mathcal P$. Since we assume that $C$ has no direct summand in $\K$, it is a direct summand of $A$. Let $A=C\oplus D$, then we have $R_2\xrightarrow{g=\svecv{g_1}{g_2}} C\oplus D\xrightarrow{h=\svech{h_1}{h_2}} T_1$. We also have $C\xrightarrow{v=\svecv{v_1}{v_2}} C\oplus D\xrightarrow{\svech{v_1'}{v_2'}} D$ such that $v_1'v_1+v_2'v_2=1_C$ where $v_1$ is an isomorphism, then $\svecv{g_1}{g_2}-\svecv{v_1}{v_2}u'=\svecv{q_1}{q_2}r$. Let $v_1''$ be the inverse of $v_1$, since $\svech{-v_2v_1''}{1_D}\svecv{g_1}{g_2}=g_2-v_2v_1''g_1=q_2r+v_2u'-v_2v_1''g_1=q_2r+v_2v_1''(v_1u'-g_1)=q_2r-v_2v_1''q_1r=\svech{-v_2v_1''}{1_D}\svecv{q_1}{q_2}r$. Hence we have the following commutative diagram
$$\xymatrix{
R_2 \ar[d]_-{\svecv{g_1}{g_2}} \ar[rr]^{r} &&P \ar[d] \ar[rdd] \\
C\oplus D \ar[rr]^-{h=\svech{h_1}{h_2}} \ar[rrrd]_{\svech{-v_2v_1''}{1_D}}  &&T_1 \ar@{.>}[rd]\\
&&&D
}
$$
which implies $D$ is a direct summand of $T_1$. Let $T_1= E\oplus D$, the we have $$C\oplus D\xrightarrow{\svech{h_1}{h_2}=\left(\begin{smallmatrix}
h_{11}&d_1\\
h_{21}&d_2
\end{smallmatrix}\right)} E\oplus D \xrightarrow{t=\svech{t_1}{t_2}} T_2$$ where $d_2$ is an isomorphism. Then we have an isomorphism $$E\oplus D \xrightarrow{\left(\begin{smallmatrix}
1_E&d_1\\
0&d_2
\end{smallmatrix}\right)} E\oplus D$$ such that $\svech{t_1}{t_2}\left(\begin{smallmatrix}
1_E&d_1\\
0&d_2
\end{smallmatrix}\right)=\svech{t_1}{0}.$ But $\underline t$ is right minimal, it implies $H(D)=0$. This means $H(\underline v)$ is an isomorphism, hence $H(\underline g)$ is an epimorphism and $H(\underline h)=0$.
\smallskip

(IV) We show full and dense implies (b).

\smallskip
Let $A\notin \K$ be an indecomposable object, then it admits an exact sequence $$H(R_1) \xrightarrow{H(\underline f)} H(R_2) \xrightarrow{H(\underline g)} H(A)\to 0.$$ If $H(\underline f)=0,$  $H(\underline g)$ becomes an isomorphism, $H(g)\neq 0$, hence $R_2$ is an direct summand of $A\oplus P'$, where $P'\in \mathcal P$. Since $A$ is indecomposable and $R_2\notin \K$, $A$ is a direct summand of $R_2$, then by Lemma \ref{summand}, $A\in \Omega \T$. Hence it admits a commutative diagram
$$\xymatrix{
0 \ar@{=}[r] \ar[d] &0 \ar[d]\\
A \ar@{=}[d] \ar[r]^h &P_A \ar[r] \ar@{=}[d] &T_A \ar@{=}[d] \ar@{-->}[r]&\\
A \ar[r]_h \ar@{-->}[d] &P_A \ar[r] \ar@{-->}[d] &T_A \ar@{-->}[r] &\\
&& &&
}
$$
where $H(\underline h)=0$. Otherwise, we can assume that $H(\underline f)\neq 0$ is right minimal, hence $0\neq \underline f$ is right minimal.By Lemma \ref{tec}, there is a commutative diagram
$$\xymatrix{
R_1 \ar@{=}[r] \ar[d]_{f'} &R_1\ar[d]\\
R_2' \ar[d]_{g'} \ar[r] &P \ar[r] \ar[d] &T_2 \ar@{=}[d] \ar@{-->}[r]&\\
B \ar[r]_h \ar@{-->}[d] &T_1 \ar[r] \ar@{-->}[d] &T_2 \ar@{-->}[r] &\\
&& &&
}
$$
where $P\in \mathcal P$, $\underline {f'}=\underline f$, $T_i\in \T$ and $R'_2=R_2$ in $\underline B$. By (a), we have $H(\underline h)=0$, then we have a short exact sequence $H(R_1)\xrightarrow{H(\underline {f})} H(R_2)\xrightarrow{H(\underline {g'})} H(B)\to 0
$. Hence we have the following commutative diagram
$$\xymatrix{
H(R_2) \ar[r]^{H(\underline {g'})} \ar@{=}[d] &H(B) \ar@{.>}[d]^{H(\underline a)}_{\simeq} \ar[r] &0\\
H(R_2) \ar[r]^{H(\underline g)} \ar@{=}[d] &H(A) \ar@{.>}[d]^{H(\underline b)}_{\simeq} \ar[r] &0\\
H(R_2) \ar[r]^{H(\underline {g'})}  &H(B) \ar[r] &0\\
}
$$
Since $A$ is indecomposable, $ab$ must be an isomorphism. which implies $A$ is a direct summand of $B$. We also have $\underline {g'}=\underline {bg}$. Then the rest argument is similar as (III).
\end{proof}

When we consider $\B^{\op}$, we have the following theorem, which is the dual of Theorem \ref{main1}.

\begin{thm}\label{mainop}
Let $\T$ be a covariantly finite subcategory such that $\mathcal I\subset \T$ and $\EE(\mathcal I,\T)=0$, the functor $\Hom_{\overline \B}(-,\Sigma \T)$ is a quotient functor, if and only if the following conditions are satisfied:
\begin{itemize}
\item[(a$^{\op}$)] In the following commutative diagram
$$\xymatrix{
T_2 \ar[r] \ar@{=}[d] &T_1 \ar[d] \ar[r]^h &A \ar[d] \ar@{-->}[r] &\\
T_2 \ar[r] &I \ar[r] \ar[d] &S_2 \ar[d]_g \ar@{-->}[r] &\\
&S_1 \ar@{=}[r] \ar@{-->}[d] &S_1 \ar@{-->}[d]\\
&& &&}
$$
where $T_1,T_2\in \T$ and $I\in \mathcal I$, if $0\neq \overline g$ if left minimal, then $\Hom_{\overline \B}(\overline h,\Sigma T)=0$.
\item[(b$^{\op}$)] For any indecomposable object $A$ such that $\Hom_{\overline \B}(A,\Sigma \T) \neq 0$, $A$ admits a commutative diagram
$$\xymatrix{
T_2 \ar[r] \ar@{=}[d] &T_1 \ar[d] \ar[r]^h &A \ar[d] \ar@{-->}[r] &\\
T_2 \ar[r] &I \ar[r] \ar[d] &S_2 \ar[d]_g \ar@{-->}[r] &\\
&S_1 \ar@{=}[r] \ar@{-->}[d] &S_1 \ar@{-->}[d]\\
&& &&}
$$
where $T_1,T_2\in \T$, $I\in \mathcal I$ and $\Hom_{\overline \B}(\overline h,\Sigma \T)=0$.
\end{itemize}

\end{thm}

We conclude this section with two examples illustrating Theorem \ref{main1}.

\begin{exm}\label{ex1}
Let $\Lambda$ be the the $k$-algebra given by the quiver
$$\xymatrix@C=0.5cm@R0.5cm{
\centerdot \ar[r]^x &\centerdot \ar[r]^x &\centerdot \ar[r]^x &\centerdot \ar[r]^x &\centerdot \ar[r]^x &\centerdot \ar[r]^x &\centerdot
}
$$
with relation $x^3=0$. Then the AR-quiver of $\B:=\mod\Lambda$ is given by
$$
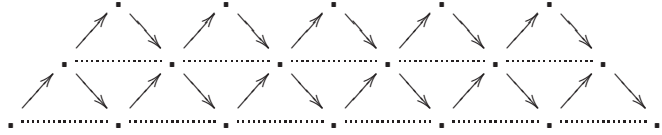
\xymatrix@C=0.4cm@R0.4cm{
 &&\centerdot \ar[dr] &&\centerdot \ar[dr] &&\centerdot \ar[dr] &&\centerdot \ar[dr] &&\centerdot \ar[dr]\\
 &\centerdot \ar[dr] \ar@{.}[rr] \ar[ur] &&\centerdot \ar@{.}[rr] \ar[dr] \ar[ur] &&\centerdot \ar[dr] \ar[ur] \ar@{.}[rr] &&\centerdot \ar@{.}[rr] \ar[dr] \ar[ur] &&\centerdot \ar[dr] \ar[ur] \ar@{.}[rr] && \centerdot \ar[dr]\\
\centerdot \ar@{.}[rr] \ar[ur] &&\centerdot \ar@{.}[rr] \ar[ur] &&\centerdot \ar@{.}[rr] \ar[ur] &&\centerdot \ar@{.}[rr] \ar[ur] &&\centerdot \ar@{.}[rr] \ar[ur] &&\centerdot \ar[ur] \ar@{.}[rr] &&\centerdot
}
$$
We denote by ``$\circ$" in the AR-quiver the indecomposable objects belong to a subcategory and by ``$\cdot$'' the indecomposable objects do not belong to it. Let
$$
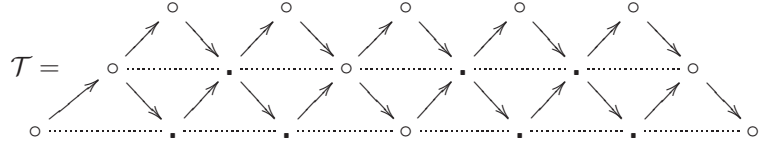
\xymatrix@C=0.4cm@R0.4cm{
 &&\circ \ar[dr] &&\circ \ar[dr] &&\circ \ar[dr] &&\circ \ar[dr] &&\circ \ar[dr]\\
\T= &\circ \ar[dr] \ar@{.}[rr] \ar[ur] &&\centerdot \ar@{.}[rr] \ar[dr] \ar[ur] &&\circ \ar[dr] \ar[ur] \ar@{.}[rr] &&\centerdot \ar@{.}[rr] \ar[dr] \ar[ur] &&\centerdot \ar[dr] \ar[ur] \ar@{.}[rr] && \circ \ar[dr]\\
\circ \ar@{.}[rr] \ar[ur] &&\centerdot \ar@{.}[rr] \ar[ur] &&\centerdot \ar@{.}[rr] \ar[ur] &&\circ \ar@{.}[rr] \ar[ur] &&\centerdot \ar@{.}[rr] \ar[ur] &&\centerdot \ar[ur] \ar@{.}[rr] &&\circ
}
$$
we can get
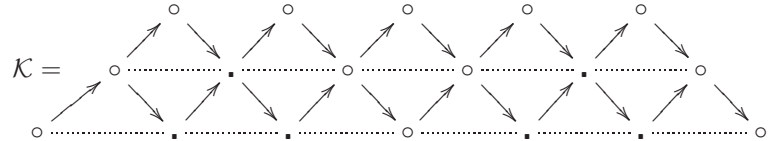
$$\xymatrix@C=0.4cm@R0.4cm{
 &&\circ \ar[dr] &&\circ \ar[dr] &&\circ \ar[dr] &&\circ \ar[dr] &&\circ \ar[dr]\\
\K= &\circ \ar[dr] \ar@{.}[rr] \ar[ur] &&\centerdot \ar@{.}[rr] \ar[dr] \ar[ur] &&\circ \ar[dr] \ar[ur] \ar@{.}[rr] &&\circ \ar@{.}[rr] \ar[dr] \ar[ur] &&\centerdot \ar[dr] \ar[ur] \ar@{.}[rr] && \circ \ar[dr]\\
\circ \ar@{.}[rr] \ar[ur] &&\centerdot \ar@{.}[rr] \ar[ur] &&\centerdot \ar@{.}[rr] \ar[ur] &&\circ \ar@{.}[rr] \ar[ur] &&\centerdot \ar@{.}[rr] \ar[ur] &&\centerdot \ar[ur] \ar@{.}[rr] &&\circ
}
$$
Since in this case $\T$ is rigid, in fact $(\T,\K)$ is a cotorsion pair (see \cite{NP} for the definition of a cotorsion pair). Then $H$ is a quotient functor if and only if $\B/\K$ is abelian (see Proposition \ref{eq2}). It is the case in this example, and we have $\B/\K\simeq \mod \underline {\Omega \T}$.
\end{exm}

\begin{exm}\label{ex2}
Let $\Lambda$ be the the $k$-algebra given by the quiver
$$\xymatrix@C=0.5cm@R0.5cm{
\centerdot \ar[r]^x &\centerdot \ar[r]^x &\centerdot \ar[r]^x &\centerdot \ar[r]^x &\centerdot \ar[r]^x &\centerdot
}
$$
with relation $x^3=0$. Then the AR-quiver of $\B:=\mod\Lambda$ is given by
$$
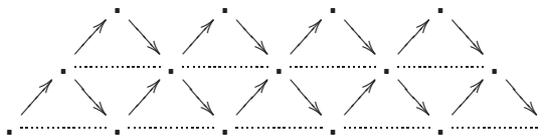
\xymatrix@C=0.4cm@R0.4cm{
 &&\centerdot \ar[dr] &&\centerdot \ar[dr] &&\centerdot \ar[dr] &&\centerdot \ar[dr] \\
&\centerdot \ar[dr] \ar@{.}[rr] \ar[ur] &&\centerdot \ar@{.}[rr] \ar[dr] \ar[ur] &&\centerdot \ar[dr] \ar[ur] \ar@{.}[rr] &&\centerdot \ar@{.}[rr] \ar[dr] \ar[ur] &&\centerdot \ar[dr]  \\
\centerdot \ar@{.}[rr] \ar[ur] &&\centerdot \ar@{.}[rr] \ar[ur] &&\centerdot \ar@{.}[rr] \ar[ur] &&\centerdot \ar@{.}[rr] \ar[ur] &&\centerdot \ar@{.}[rr] \ar[ur] &&\centerdot
}
$$
We denote by ``$\circ$" in the AR-quiver the indecomposable objects belong to a subcategory and by ``$\cdot$'' the indecomposable objects do not belong to it. Let
$$
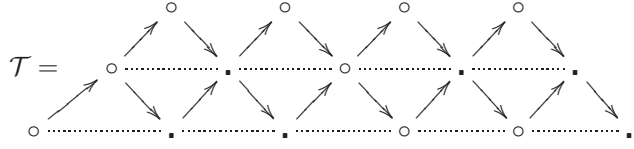
\xymatrix@C=0.4cm@R0.4cm{
 &&\circ \ar[dr] &&\circ \ar[dr] &&\circ \ar[dr] &&\circ \ar[dr] \\
\T= &\circ \ar[dr] \ar@{.}[rr] \ar[ur] &&\centerdot \ar@{.}[rr] \ar[dr] \ar[ur] &&\circ \ar[dr] \ar[ur] \ar@{.}[rr] &&\centerdot \ar@{.}[rr] \ar[dr] \ar[ur] &&\centerdot \ar[dr]  \\
\circ \ar@{.}[rr] \ar[ur] &&\centerdot \ar@{.}[rr] \ar[ur] &&\centerdot \ar@{.}[rr] \ar[ur] &&\circ \ar@{.}[rr] \ar[ur] &&\circ \ar@{.}[rr] \ar[ur] &&\centerdot
}
$$
then we have
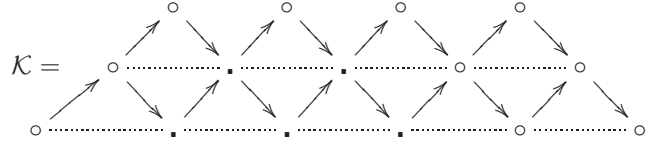
$$\xymatrix@C=0.4cm@R0.4cm{
 &&\circ \ar[dr] &&\circ \ar[dr] &&\circ \ar[dr] &&\circ \ar[dr] \\
\K= &\circ \ar[dr] \ar@{.}[rr] \ar[ur] &&\centerdot \ar@{.}[rr] \ar[dr] \ar[ur] &&\centerdot \ar[dr] \ar[ur] \ar@{.}[rr] &&\circ \ar@{.}[rr] \ar[dr] \ar[ur] &&\circ \ar[dr]  \\
\circ \ar@{.}[rr] \ar[ur] &&\centerdot \ar@{.}[rr] \ar[ur] &&\centerdot \ar@{.}[rr] \ar[ur] &&\centerdot \ar@{.}[rr] \ar[ur] &&\circ \ar@{.}[rr] \ar[ur] &&\circ
}
$$
In this example we can find that any indecomposable object $A$ admits a short exact sequence $0\to A\xrightarrow{a} T^1\to T^2 \to 0$ such that $T^1,T^2\in \T$ and $H(\underline a)=0$. For any short exact sequence $0\to B\xrightarrow{b} T_1\to T_2 \to 0$ where $T_1,T_2\in \T$ we also have $H(\underline b)=0$. By Theorem \ref{main1}, $H$ is a quotient functor.
\end{exm}

\section{The Second Main Result }

From this section, let $\T$ be a functorially finite subcategory of $\B$ such that
\begin{itemize}
\item $\mathcal P\subset \T$ and $\mathcal I\subset \T$;
\item $\EE(\T,\mathcal P)=0=\EE(\mathcal I,\T)$.
\item ${^{\bot_1}}\T=\K$, where ${^{\bot_1}}\T$ is the subcategory of objects $X\in\B$ satisfying
$\EE(X,\T)=0$.
\end{itemize}

Under this assumption, when $(\T,\K)$ is a cotorsion pair (for definition of cotorsion pair on extriangulated category, please see \cite{NP}), we have the following equivalent condition.

\begin{prop}\label{eq2}
When $(\T,\K)$ is a cotorsion pair, the following conditions are equivalent:
\begin{itemize}
\item[(i)] $\B/\K$ is abelian.
\item[(ii)] $H$ is a quotient functor.
\end{itemize}
\end{prop}

\begin{proof}
We first show that for a morphism $f:A\to B$ where $A,B$ is indecomposable, $H(\underline f)=0$ if and only if $f$ factors through $\K$.

The ``if" part is obvious. For the ``only if" part, it is trivial if $A\in \K$. Now we assume $A\notin \K$, then $A$ admits a commutative diagram:
$$\xymatrix{
\Omega K \ar@{=}[r] \ar[d] &\Omega K \ar[d]\\
\Omega T \ar[d]_g \ar[r]^p &P \ar[r] \ar[d]^q &T
 \ar@{=}[d] \ar@{-->}[r]&\\
A \ar[r]_h \ar@{-->}[d] &K \ar[r] \ar@{-->}[d] &T \ar@{-->}[r] &\\
&& &&
}
$$
where $T\in \T$, $K\in \K$ and $P\in \mathcal P$. Then we get an $\EE$-triangle $\xymatrix{\Omega T\ar[r]^-{\svecv{g}{p}} &A\oplus P\ar[r]^-{\svech{h}{-q}} &K\ar@{-->}[r] &}$. If $H(\underline f)=0$, we get $\underline {fg}=0$, then $fg$ factors through a object $P'\in \mathcal P$. Hence we have the following commutative diagram
$$\xymatrix@C=1.2cm{
\Omega T \ar[r]^-{\svecv{g}{p}} \ar[d]_{p'} &A\oplus P\ar[r]^{\svech{h}{-q}} \ar[d]^-{\svech{f}{0}} &K\ar@{-->}[r] &\\
P' \ar[r]_{q'} &B
}
$$
Since $\EE(K,P)=0$, there is a morphism $A\oplus P\xrightarrow{\svech{a}{b}}P'$ such that $p'=\svech{a}{b}\svecv{g}{p}$. Hence $$\svech{f-q'a,
}{-q'b}\svecv{g}{q}=0,$$ and there is a morphism $k:K\to B$ such that $f-q'a=kh$. This means $\underline f$ factors through $K\in \K$.
(ii)$\Rightarrow$(i): if $H$ is a quotient functor, we get the following commutative diagram as Remark \ref{ideal}
$$\xymatrix@C=0.5cm@R0.5cm{
\B \ar[rr]^H \ar[dr]_{\pi} && \mod \underline{\Omega \T}\\
&\B/\K \ar[ur]_-{\simeq}
}
$$
which implies $\B/\K$ is abelian.\\
(i)$\Rightarrow$(ii): it is enough to show that $\Omega \T$ is the subcategory of enough projetives in $\B/\K$. For convenience, here we denote the image of a morphism $f$ in $\B/\K$ by $\overline f$.

We claim that a morphism $\overline f:A\to B$ is an epimorphism if and only if we have the following commutative diagram
$$\xymatrix{
A \ar[r] \ar[d]_f & K \ar[d] \ar[r] &T  \ar@{=}[d] \ar@{-->}[r] &\\
B \ar[r]_b &C \ar[r] &T \ar@{-->}[r] &
}
$$
where $K\in \K$, $T\in T$ and $\overline b=0$.\\
We show ``if" part first.\\
If $\overline f:A\to B$ is an epimorphism, then it admits the following commutative diagram
$$\xymatrix{
A \ar[r] \ar[d]_f & K \ar[r] \ar[d] &T \ar@{=}[d] \ar@{-->}[r] &\\
B \ar[r]_b &C \ar[r] &T \ar@{-->}[r] &
}
$$
where $K\in \K$ and $T\in T$. Then we have $\overline {bf}=0$. Since $\overline f$ is epic, we have $\overline b=0$.\\
Now we show ``only if" part.\\
If we have the following commutative diagram
$$\xymatrix{
A \ar[r]^k \ar[d]_f & K \ar[d] \ar[r] &T  \ar@{=}[d] \ar@{-->}[r] &\\
B \ar[r]_b &C \ar[r] &T \ar@{-->}[r] &
}
$$
where $K\in \K$, $T\in T$ and $\overline b=0$. Let $d:B\to D$ be a morphism such that $\overline {df}=0$, then $df$ factors through an object in $\K$, it must factor through $k$. Hence we have the following commutative diagram
$$\xymatrix{
A \ar[r]^k \ar[d]_f &K \ar[d] \ar[ddr]\\
B \ar[r]^b \ar[drr]_d &C \ar@{.>}[dr]\\
&&D
}
$$
Hence $\overline d=0$, which implies $\overline f$ is epic.\\
Now let $\overline f:A\to B$ be an epimorphism and $t:\Omega T\to B$ be a morphism where $\Omega T_0$ admits an $\EE$-triangle $\xymatrix{\Omega T_0 \ar[r]^{p_0} &P_0 \ar[r] & T_0 \ar@{-->}[r] &}$ where $P_0\in \mathcal P$ and $T_0\in \T$. By the argument above, we have the following commutative diagram
$$\xymatrix{
A \ar[r]^k \ar[d]_f & K \ar[d]^c \ar[r] &T  \ar@{=}[d] \ar@{-->}[r] &\\
B \ar[r]_b &C \ar[r] &T \ar@{-->}[r] &
}
$$
where $K\in \K$, $T\in \T$ and $\overline b=0$. Then we get an $\EE$-triangle $\xymatrix{A \ar[r]^-{\svecv{f}{k}} &B\oplus K\ar[r]^-{\svech{b}{-c}} &C\ar@{-->}[r] &}$. Since $\overline {bt}=0$, then it factors through $p_0$, we have the following commutative diagram
$$\xymatrix{
&\Omega T_0 \ar[r]^{p_0} \ar[d]^-{\svecv{t}{0}} &P_0 \ar[r] \ar[d] & T_0 \ar@{-->}[r] &\\
A \ar[r]_-{\svecv{f}{k}} &B\oplus K\ar[r]_-{\svech{b}{-c}} &C\ar@{-->}[r] &
}
$$
By the similar argument as in Proposition \ref{abelian}, we get $\overline t$  factors through $\overline f$. By the previous argument in this proof, any indecomposable object $A'\notin \K$ admits the following commutative diagram
$$\xymatrix{
\Omega K \ar@{=}[r] \ar[d] &\Omega K \ar[d]\\
\Omega T \ar[d]_g \ar[r]^p &P \ar[r] \ar[d]^q &T
 \ar@{=}[d] \ar@{-->}[r]&\\
A' \ar[r]_h \ar@{-->}[d] &K \ar[r] \ar@{-->}[d] &T \ar@{-->}[r] &\\
&& &&
}
$$
which implies $\overline g:\Omega T\to A'$ is an epimorphism. Hence $\Omega \T$ is the subcategory of enough projective objects in $\B/\K$.
\end{proof}

Let $\widetilde{\K}$ be the subcategory of $\K$ that any object in it does not have direct summand in $\T$.
The following lemma is useful, for proof, see \cite{LZ}.

\begin{lem}\label{app}
Let $\D$ be a rigid subcategory, $\xymatrix{D \ar[r]^s &C \ar[r]^t &K \ar@{-->}[r] &}$ be an $\EE$-triangle where $K\in \widetilde{\K}$ and $t$ be a right $\mathcal T$-approximation (resp. $\mathcal P$, $\mathcal I$-approximation). Then
\begin{itemize}
\item[(a)] if $K$ is indecomposable, we have $D=C_0\oplus X$ where $C_0\in \mathcal T$ (resp. $\mathcal P$, $\mathcal I$) and $X$ is indecomposable and $X$ does not belong to $\T$. Moreover, if $t$ is right minimal, $D$ is indecomposable.
\item[(b)] $K$ is indecomposable if $D$ is indecomposable.
\end{itemize}
\end{lem}

We are considering when $\T$ becomes a cluster-tilting subcategory under certain conditions as in Theorem \ref{main1}. The following proposition is an answer.

\begin{prop}\label{CT}
The subcategory $\T$ is cluster tilting if and only if the following conditions are satisfied:
\begin{itemize}
\item[(a)] In the following commutative diagram
$$\xymatrix{
R_1 \ar@{=}[r] \ar[d]_f &R_1 \ar[d]\\
R_2 \ar[d]_g \ar[r] &P \ar[r] \ar[d] &T_2 \ar@{=}[d] \ar@{-->}[r]&\\
A \ar[r]_h \ar@{-->}[d] &T_1 \ar[r] \ar@{-->}[d] &T_2 \ar@{-->}[r] &\\
&& &&
}
$$
where $T_1,T_2\in \T$ and $P\in \mathcal P$, if $0\neq \underline f$ is right minimal, then $H(\underline h)=0$.
\item[(b*)] Any indecomposable object $A$ admits a commutative diagram
$$\xymatrix{
R_1 \ar@{=}[r] \ar[d]_f &R_1 \ar[d]\\
R_2 \ar[d]_g \ar[r] &P \ar[r] \ar[d] &T_2 \ar@{=}[d] \ar@{-->}[r]&\\
A \ar[r]_h \ar@{-->}[d] &T_1 \ar[r] \ar@{-->}[d] &T_2 \ar@{-->}[r] &\\
&& &&
}
$$
where $T_1,T_2\in \T$, $P\in \mathcal P$ and $H(\underline h)=0$.
\item[(c)] $\T\cap \Omega \T=\mathcal P$.
\end{itemize}
\end{prop}

\begin{proof}
Suppose (a), (b*), (c) are satisfied, we show $\T$ is cluster tilting.

We first prove $\T$ is rigid.

Let $A$ be an indecomposable object in $\T$, by (b*), it admits a commutative diagram
$$\xymatrix{
R_1 \ar@{=}[r] \ar[d]_f &R_1 \ar[d]\\
R_2 \ar[d]_g \ar[r]^q &P \ar[r] \ar[d]^p &T_2 \ar@{=}[d] \ar@{-->}[r]&\\
A \ar[r]_h \ar@{-->}[d] &T_1 \ar[r] \ar@{-->}[d] &T_2 \ar@{-->}[r] &\\
&& &&
}
$$
where $T_1,T_2\in \T$, $P\in \mathcal P$ and $\Hom_{\underline \B}(\Omega \T,\underline h)=0$.  Since A admits an $\EE$-triangle $\xymatrix@C=0.7cm{\Omega A \ar[r] &P_A \ar[r]^{p_A}  &A \ar@{-->}[r]&}$ where $P_A\in \mathcal P$. Then we have the following commutative diagram
$$\xymatrix{
\Omega A \ar@{=}[r] \ar[d]_{h_1} &\Omega A \ar[d] \\
R_1' \ar[r] \ar[d]_{f'} &P\oplus P_A \ar[d]^{\left(\begin{smallmatrix}
1_P&0\\
0&p_A
\end{smallmatrix}\right)} \ar[r] &T_1 \ar@{=}[d] \ar@{-->}[r]^{\sigma} &\\
R_2 \ar[r]_-{\svecv{-q}{g}} \ar@{-->}[d] &P\oplus A \ar[r]_-{\svech{p}{h}} \ar@{-->}[d] &T_1 \ar@{-->}[r]^{\delta} &\\
&& &&
}
$$
Moreover, we also have the following commutative diagram according to (ET4)${^{\op}}$ in \cite{NP}.
$$\xymatrix{
\Omega A \ar[r] \ar[d]_{h_1} &P\oplus P_A \ar[r] \ar[d] &A\ar[d]^h\ar@{-->}[r]^{\sigma} &\\
R_1' \ar[r] &P\oplus P_A \ar[r] &T^1\ar@{-->}[r]^{\delta} &
}
$$
By Lemma \ref{app}, $\Omega A=R_A\oplus P'$ where $R_A$ is an indecomposable object and $P'\in \mathcal P$. If $\underline {h_1}=0$, $h$ factors through $P\oplus P_A$, then we have the following diagram
$$\xymatrix{
&A \ar[r] \ar@{.>}[dl] \ar[d]^{\svecv{0}{1_A
}} &P\oplus P_A \ar[d] \ar@{.>}[dl] \\
R_2 \ar[r] &P\oplus A \ar[r]_-{\svech{p}{h}} &T_1 \ar@{-->}[r] &
}
$$
which implies $A$ is a direct summand of $R_2\oplus P\oplus P_A$, then A  lies in $\Omega \T$ by Lemma \ref{summand}, by (c) $A\in \mathcal P$, hence $H(A)=0$. Now we can assume $\underline {h_1}\neq 0$, it is right minimal since $\Omega A=R_A$ is indecomposable in $\underline B$. By (a), we have $H(\underline {g})=0$. Hence by Lemma \ref{exact} we have an exact sequence $H(R_2) \xrightarrow{H(\underline g)=0} H(A) \xrightarrow{H(\underline h)=0} H(T_1)$ which implies $H(A)=0$, hence $\T$ is rigid.

Now let $A$ be an indecomposable object such that $H(A)=0$, we show that $A\in \T$.\\
By (b*), $A$ admits a commutative diagram
$$\xymatrix{
R_1 \ar@{=}[r] \ar[d]_f &R_1 \ar[d]\\
R_2 \ar[d]_g \ar[r]^r &P \ar[r] \ar[d] &T_2 \ar@{=}[d] \ar@{-->}[r]&\\
A \ar[r]_h \ar@{-->}[d] &T_1 \ar[r]_t \ar@{-->}[d] &T_2 \ar@{-->}[r] &\\
&& &&
}
$$
where $T_1,T_2\in \T$, $P\in \mathcal P$. By assumption we have $H(\underline g)=0\Rightarrow \Hom_{\underline \B}(R_2,A)=0\Rightarrow \underline g=0$, hence it factors through  $r$. Then $t$ is a split epimorphism and $A$ is a s direct summand of $T^1$, hence $A\in \T$.

Now by definition $\T$ is a cluster tilting subcategory.

If $\T$ is a cluster tilting subcategory, (a), (b*), (c) are satisfied by the definition of cluster tilting.
\end{proof}

Let $\h$ be the subcategory of direct sums of indecomposable objects $X\notin \K$ and $Y\in \T$, then an indecomposable object in $\B$ belongs to either $\h$ or $\widetilde{\K}$.

\begin{lem}\label{stable}
If $\T$ satisfies condition (b) in Theorem \ref{main1}, then any object $X$ in $\h$ admits a commutative diagram.
$$\xymatrix{
R_1 \ar@{=}[r] \ar[d]_f &R_1 \ar[d]\\
R_2 \ar[d]_g \ar[r] &P \ar[r] \ar[d]^p &T_2 \ar@{=}[d] \ar@{-->}[r]& \ar@{}[d]^{(\star)}\\
X \ar[r]_h \ar@{-->}[d] &T_1 \ar[r] \ar@{-->}[d] &T_2 \ar@{-->}[r] &\\
&& &&
}
$$
where $T_1,T_2\in \T$, $P\in \mathcal P$.
\end{lem}

\begin{proof}
Let $X\in \h$ be an indecomposable object. If $X\notin \K$, by (b) we have the required diagram. Otherwise let $X\in \K\cap \T$, then we have the following diagram
$$\xymatrix{
\Omega X \ar@{=}[r] \ar[d] &\Omega X \ar[d]\\
P_X \ar[d] \ar@{=}[r] &P_X \ar[r] \ar[d] &0 \ar@{=}[d] \ar@{-->}[r]&\\
X \ar@{=}[r] \ar@{-->}[d] &X \ar[r] \ar@{-->}[d] &0 \ar@{-->}[r] &\\
&& &&
}
$$
where $H(\underline {1_X})=0$ since $H(X)=0$. Hence any object $X$ in $\h$ admits a commutative diagram as ($\star$).
%If $\underline h=0$, then $h$ factors through $P$, which implise $X$ is a direct summand of $R_2$, hence $X\in \Omega \T$. By definition $Hom_{\underline B}(X,Y)=0$.
\end{proof}

By this lemma and the proof of Proposition \ref{CT}, we have the following corollary.
\medskip

\begin{cor}\label{rigid}
If $\T$ satisfies condition (b) in Theorem \ref{main1} and $\T\cap \Omega \T=\mathcal P$, then $\T$ is rigid.
\end{cor}

Let $\widehat{\K}$ be the subcategory of objects which are direct sums of objects in $\widetilde{\K}$ and $\mathcal P\cap \mathcal I$.

\begin{lem}\label{zero}
If $H$ is a quotient functor, then $\Hom_{\underline B}(\h,\widetilde{\K})=0$.
\end{lem}

\begin{proof}
Let $x:X\to Y$ be a morphism where $X\in \h$ and $Y\in \widetilde{\K}$ are indecomposable objects. By Theorem \ref{main1}, $X$ admits a commutative diagram
$$\xymatrix{
R_1 \ar@{=}[r] \ar[d]_f &R_1 \ar[d]\\
R_2 \ar[d]_g \ar[r] &P \ar[r] \ar[d]^p &T_2 \ar@{=}[d] \ar@{-->}[r]& \ar@{}[d]^{(\star)}\\
X \ar[r]_h \ar@{-->}[d] &T_1 \ar[r] \ar@{-->}[d] &T_2 \ar@{-->}[r] &\\
&& &&
}
$$
where $T_1,T_2\in \T$, $P\in \mathcal P$. Since $\EE(\T,\K)=0$, we have $x:X\xrightarrow{h} T_1\xrightarrow{t_1} Y$. If $\underline x\neq 0$, we have $\Omega Y\notin \K/\mathcal P$ since $\underline {t_1}\neq 0$. By Lemma \ref{app}, $\Omega Y=R_Y\oplus P'$ where $R_Y\in \Omega \widehat{\K}$ is indecomposable and $P'\in \mathcal P$. By Theorem \ref{main1}, $\Omega Y$ admits a commutative diagram
$$\xymatrix{
R_1' \ar@{=}[r] \ar[d]_{f'} &R_1' \ar[d]\\
R_2' \ar[d]_{g'}\ar[r] &P' \ar[r] \ar[d]^r &{T_2}' \ar@{=}[d] \ar@{-->}[r]& \\
\Omega Y \ar[r]_{h'} \ar@{-->}[d] &{T_1}' \ar[r] \ar@{-->}[d] &{T_2}' \ar@{-->}[r] &\\
&& &&
}
$$
We have the following commutative diagram
$$\xymatrix{
\Omega Y \ar[r] \ar[d]_{h'} &P_Y \ar[r] \ar[d] &Y \ar[d]^{h''} \ar@{-->}[r] &\\
{T_1}' \ar[r] &I \ar[r]_i &\Sigma {T_1}' \ar@{-->}[r] &
}
$$
where $I\in \mathcal I$, If $\overline {h''}=0$, then it factors through $i$, hence $h'$ factors through $P_Y$ and then factors through $r$, which implies $R_Y$ is a direct summand of $R_2'$, hence $R_Y\in \Omega \T$, a contradiction. If $\overline {h''}\neq 0$, we have $H(Y)\neq 0$, then $Y\notin \K$, a contradiction. Hence $\underline x=0$.
\end{proof}

From now on, we assume $\B$ satisfies condition (WIC) (\cite[Condition 5.8]{NP}):

\begin{itemize}
\item If we have a deflation $h: A\xrightarrow{f} B\xrightarrow{g} C$, then $g$ is also a deflation.
\item If we have an inflation $h: A\xrightarrow{f} B\xrightarrow{g} C$, then $f$ is also an inflation.
\end{itemize}

Note that this condition holds on triangulated categories and exact categories. If $\T$ is extension closed, then under this condition we can get cotorsion pairs $(\T,\K)$ and $(\K,\T)$.

By this condition, we can always get right minimal deflations and left minimal inflations.

Denote $\Hom_{\overline \B}(-,\Sigma \T)$ by $H^{\op}$, we have the following lemma.

\begin{lem}\label{syzygy}
If $H$ and $H^{\op}$ are quotient functors, then any indecomposable object $K\in \widetilde{\K}$ admits the following $\EE$-triangles
\begin{itemize}
\item[(a)] $\xymatrix{K' \ar[r] &P \ar[r] &K \ar@{-->}[r] &}$,
\item[(b)] $\xymatrix{K \ar[r] &I \ar[r] &K'' \ar@{-->}[r] &}$
\end{itemize}
where $P,I \in \mathcal P \cap \mathcal I$ and $K',K''\in \widetilde{\K}$ are indecomposables.
\end{lem}

\begin{proof}
Let $K\in \widetilde{\K}$ be an indecomposable object, $K$ admits an $\EE$-triangle $\xymatrix{\Omega K \ar[r] &P \ar[r]^p  &K \ar@{-->}[r]&}$ where $P\in \mathcal P$ and $p$ is right minimal. By Lemma \ref{app}, $\Omega K$ is indecomposable and $P\in \mathcal P$.

We first show $\Omega K\in \widetilde{\K}$.

Obviously $\Omega K\notin \T$, if $\Omega K\notin \widetilde{\K}$,  it admits an $\EE$-triangle $\xymatrix{\Omega K \ar[r] &T \ar[r] &T' \ar@{-->}[r] &}$ where $T,T'\in \T$, then we have the following commutative diagram.
$$\xymatrix{
\Omega K \ar@{=}[d]\ar[r] &T \ar[d] \ar[r] &T' \ar[d]^t \ar@{-->}[r] &\\
\Omega K \ar[r] &P \ar[r]^p  &K \ar@{-->}[r]&
}
$$
By Lemma \ref{zero}, $t$ factors through $\mathcal P$, hence factors through $p$, which implies $\Omega K$ is a direct summand of $T$, which means $\Omega K\in \T$, a contradiction.\\
We denote $\Omega K$ by $K'$, by the dual of the argument above, we have the following $\EE$-triangle $$\xymatrix{K' \ar[r] &I' \ar[r] &K'' \ar@{-->}[r] &}$$ where $I'\in \mathcal I$ and $K''\in \widetilde{\K}$ is indecomposable. Now we have the following commutative diagram:
$$\xymatrix{
K' \ar[r] \ar@{=}[d] &P \ar[r] \ar[d]^{i'} &K \ar[d]^{k'} \ar@{-->}[r] &\\
K' \ar[r] \ar@{=}[d] &I' \ar[r] \ar[d]^i &K'' \ar[d]^k \ar@{-->}[r] &\\
K' \ar[r] &P \ar[r]  &K \ar@{-->}[r]&
}
$$
We get $K$ is a direct summand of $K''\oplus P$, since $K\in \widetilde{\K}$, we have $K\simeq K''$ and $k'$ is an isomorphism. Hence $i'$ is also an isomorphism. Then $P\in \mathcal P\cap \mathcal I$ and (a) holds. By the same method we can show that (b) also holds.
\end{proof}

\begin{cor}\label{maincor}
If $H$ and $H^{\op}$ are quotient functors, then
 $\B/(\mathcal P\cap \mathcal I)=\widehat{\K}/(\mathcal P\cap \mathcal I)\oplus \h/(\mathcal P\cap \mathcal I)$. Moreover, $\widehat{\K}/(\mathcal P\cap \mathcal I)$ and $\h/(\mathcal P\cap \mathcal I)$ are extriangulated subcategories of $\B/(\mathcal P\cap \mathcal I)$.
\end{cor}

\begin{proof}
By Lemma \ref{stable}, its dual and Lemma \ref{syzygy}, we have $\B/(\mathcal P\cap \mathcal I)=\widehat{\K}/(\mathcal P\cap \mathcal I)\oplus \h/(\mathcal P\cap \mathcal I)$.\\
Now it is enough to show $\widehat{\K}$ and $\h$ are extension closed.\\
Let $\xymatrix{X\ar[r]^{x} &Y\ar[r]^{y} &Z \ar@{-->}[r] &}$ be an $\EE$-triangle where $X,Z\in \widehat{\K}$. We already have $Y\in \K$. If $Y$ has a direct summand $T \in \T$. Since $\Hom_{\B/(\mathcal P\cap \mathcal I)}(T,Z)=0$, then $T$ is a direct summand of some $X\oplus I$ where $I\in \mathcal P\cap \mathcal I$. Hence $T\in \mathcal P\cap \mathcal I$.\\
Now let $\xymatrix{X\ar[r]^{x} &Y\ar[r]^{y} &Z \ar@{-->}[r] &}$ be an $\EE$-triangle where $X,Z\in \h$, we have $Y=Y_1\oplus Y_2$ where $Y_1\in \h$ and $Y_2\in \widetilde{\K}$. Since $\Hom_{\B/(\mathcal P\cap \mathcal I)}(Y_2,Z)=0$, $Y_2$ is a direct summand of some $X\oplus I'$ where $I'\in \mathcal P\cap \mathcal I$. Hence $Y_2=0$ and $Y\in \h$.
\end{proof}

\textbf{Now we give the proof of Theorem \ref{main}.}

\begin{proof}
Since $\T$ is functorially finite, satisfying (1),(2),(3), $H$ and $H^{\op}$ are quotient functors, by Corollary \ref{maincor}, we get $\B/(\mathcal P\cap \mathcal I)=\widehat{\K}/(\mathcal P\cap \mathcal I)\oplus \h/(\mathcal P\cap \mathcal I)$. If $\B/(\mathcal P\cap \mathcal I)$ is connected, since $\h/(\mathcal P\cap \mathcal I)\neq 0$, we have $\widetilde{\K}=0$. Then by Lemma \ref{stable} condition (b*) in Proposition \ref{CT} is satisfied.

(a) This is followed by Proposition \ref{CT}.

(b) If $\T$ is extension closed, then we can get cotorsion pairs $(\T,\K)$, $(\K,\T)$. Since $\widetilde{\K}=0$, we have $\K\subseteq \T$, hence $\K$ is rigid, then $\K$ is pre-cluster tilting \cite{LZ}. $H$ is a quotient functor implies that $\B/\K$ is an abelian category. By \cite{LZ}, $\K$ is a cluster-tilting subcategory, hence $\T=\K$ is also a cluster-tilting subcategory.
\end{proof}

Now we have the following useful corollary, which generalizes \cite[Theorem 1.10]{LZ}.

\begin{cor}\label{all abelian}
Let $\B/(\mathcal P\cap \mathcal I)$ be connected and $(\U,\V),(\V,\U)$ be cotorsion pairs on $\B$. Let $\C=\U\cap \V$ and $\K=\{ \text{direct sums of objects in }\U \text{ and objects in }\V\}$. If $\C\supset \mathcal P\cup \mathcal I$, then the following statements are equivalent
\begin{itemize}
\item[(a)] $\C$ is cluster-tilting;
\item[(b)] $\B/\K$ is abelian;
\item[(c)] $\B/\V$ is abelian;
\item[(d)] $\B/\U$ is abelian;
\item[(e)] $\B/\C$ is abelian.
\end{itemize}
\end{cor}

\begin{proof}
This is followed by Proposition \ref{eq2}, its dual and Theorem \ref{main}.
\end{proof}

Finally, we give two applications of our results in $2$-Calabi-Yau triangulated category.

\begin{cor}\label{extension closed}
Let $\B$ be a connected, Krull-Schmidt, 2-Calabi-Yau triangulated category over a field $k$, $\T$ be a functorially finite, extension closed subcategory. Then $\B/\T$ is abelian if and only if $\T$ is cluster-tilting.
\end{cor}
\begin{proof}
This is followed by Corollary \ref{all abelian}.
\end{proof}

\begin{cor}
Let $\B$ be a connected, Krull-Schmidt, Hom-finte, $k$-linear, 2-Calabi-Yau triangulated category with suspension functor $\Sigma$, $\T=\add T$ be a functorially finite subcategory. If $\Hom_{\B}(T,-)$ is full and dense, then $T$ is a cluster-tilting object or $\End_{\B}(T)^{\op}\simeq k$.
\end{cor}

\begin{proof}
If $\Hom_{\B}(T,-)$ is full and dense, functor $\Hom_{\B}(\Sigma^{-1}T,-)$ is also full and dense. 
If $T\cap \Sigma^{-1} T=0$, then $T$ is rigid by Corollary \ref{rigid}. Let $\T=\add T$, since $\B$ is Krull-Schmidt, Hom-finite, $k$-linear and $2$-Calabi-Yau, we get two cotorsion pairs $(\T,\K)$ and $(\K,\T)$. By Lemma \ref{stable}, any indecomposable object $B\notin \K$ admits a triangle $\Sigma^{-1}T^2\to B\to T^1\to T^2$ where $T^1,T^2\in T$. By \cite[Propsition 4.16(i)]{B}, we get $\B/\K$ is an abelian category. Hence by Corollary \ref{extension closed}, $\T=\K$ is a cluster tilting subcategory. Then $T$ is a cluster tilting object.

If $T\cap \Sigma T\neq 0$, then according to the proof of \cite[Theorem 27]{GJ} (note that in the proof of the case $T\cap \Sigma T\neq 0$, the assumption ``finitely many isomorphism classes of indecomposable objects" is not used), we get $\End_{\B}(T)^{\op}\simeq k$.
\end{proof}

\end{document}